\documentclass[12pt, leqno]{article}

\usepackage{amsmath,amssymb,amsthm, epsfig}

\usepackage{hyperref}

\usepackage{amsmath}

\usepackage{amssymb}

\usepackage{hyperref}

\usepackage{color}

\usepackage{ulem}

\usepackage{epsfig}

\usepackage[mathscr]{eucal}

\usepackage{mathptmx} % rm & math
\usepackage[scaled=0.90]{helvet} % ss
\usepackage{courier} % tt
\normalfont
\usepackage[T1]{fontenc}

\title{Geometric approach to nonvariational \\ singular elliptic equations}

\author{Dami\~ao Ara\'ujo  \quad $\&$  \quad Eduardo V. Teixeira }

\date{\small $\star$ Final Version $\star$ \medskip \\ $C^2$ regularity assumption on $F$-harmonic functions removed}
\newlength{\hchng}
\newlength{\vchng}
\def \dist {\mathrm{dist}}
\def \redbdry {\partial_\mathrm{red}}

\def \suchthat {\ \big | \ }

\def \Leb {\mathscr{L}^N}
\def \tr {\text{Tr}}

\newtheorem{theorem}{Theorem}[section]

\newtheorem{lemma}[theorem]{Lemma}

\newtheorem{proposition}[theorem]{Proposition}

\newtheorem{corollary}[theorem]{Corollary}

\theoremstyle{definition}

\newtheorem{definition}[theorem]{Definition}

\theoremstyle{remark}

\newtheorem{remark}[theorem]{Remark}

\numberwithin{equation}{section}

\newcommand{\intav}[1]{\mathchoice {\mathop{\vrule width 6pt height 3 pt depth  -2.5pt
\kern -8pt \intop}\nolimits_{\kern -6pt#1}} {\mathop{\vrule width
5pt height 3  pt depth -2.6pt \kern -6pt \intop}\nolimits_{#1}}
{\mathop{\vrule width 5pt height 3 pt depth -2.6pt \kern -6pt
\intop}\nolimits_{#1}} {\mathop{\vrule width 5pt height 3 pt depth
-2.6pt \kern -6pt \intop}\nolimits_{#1}}}

\begin{document}
\maketitle

\begin{abstract}
\vspace{1pc}
In this work we develop a systematic geometric approach to study fully nonlinear elliptic equations with singular absorption terms as well as their related free boundary problems. The magnitude of the singularity is measured by a negative parameter $(\gamma -1)$, for $0 < \gamma  < 1$, which reflects on  lack of smoothness for an existing solution along the singular interface between its positive and zero phases. We establish existence as well sharp regularity properties of solutions. We further prove that minimal solutions are non-degenerate and obtain fine geometric-measure properties of the free boundary $\mathfrak{F} = \partial \{u > 0 \}$. In particular we show sharp Hausdorff estimates which imply local finiteness of the perimeter of  the region $\{u > 0 \}$ and $\mathcal{H}^{n-1}$ a.e. weak differentiability property of $\mathfrak{F}$. 
\end{abstract}

%\tableofcontents

\section{Introduction}

The aim of this present work is to study fine qualitative properties of nonvariational singular elliptic equations of the form
\begin{equation} \label{Damiao_Eduardo Eq01}
\left\{
\begin{array}{rllll}
    F(D^2u)  &\sim& u^{-\theta} \cdot \chi_{\{u > 0 \}} & \mbox{in} & \Omega \\
    u&=& f & \mbox{on} & \partial\Omega,  \\
\end{array}
\right.
\end{equation}
where $\Omega \subset \mathbb{R}^N$ is a smooth bounded domain, $\theta = 1 - \gamma$, for $0<\gamma<1$,  $f$ is a positive, $C^2$ boundary datum and the  governing operator $F$ is assumed to be uniform elliptic, i.e., $\left( D_{(i,j)}F\right )_{1\le i,j \le N}$ is a positive definite matrix. The study of singular equations as \eqref{Damiao_Eduardo Eq01} is motivated by applications in a number of problems in engineering sciences. In fact the free boundary problem
\begin{equation} \label{Damiao_Eduardo Eq02}
\left\{
\begin{array}{rllll}
    F(D^2u)  & = & \gamma u^{\gamma -1} & \mbox{in} & \{u > 0\} \\
    u = |\nabla u |&=& 0 & \mbox{on} & \partial  \{u > 0\}  \\
\end{array}
\right.
\end{equation} 
is used, for example, to model fluids passing through  a porous body $\Omega$. For instance, $u$ could represent the density of a gas, or else the density of certain chemical specie, in reaction with a porous catalyst pellet, $\Omega$. 

\par

The variational theory, $F(M) = \text{Tr}(M)$, for the free boundary problem \eqref{Damiao_Eduardo Eq02} is fairly well understood, nowadays. It appears as the Euler-Lagrange equation in the minimization of non-differentiable functionals:
\begin{equation}\label{functintro}
    \int \dfrac{1}{2} |\nabla u(X)|^2 + u(X)^\gamma dX \longrightarrow \mathrm{min.} 
\end{equation}
See, for instance \cite{P01, P02, ap, Weiss}. Notice that such a problem is quite different from the one treated in the classical paper \cite{crt}. The latter has been recently studied in the fully nonlinear setting in \cite{FQS}. 

\par

The case $\gamma = 1$ in \eqref{functintro} represents the obstacle problem, \cite{C1}; the case $\gamma = 0$ relates to the cavitation problem, \cite{AC}. Fully nonlinear version of the obstacle problem has been considered in \cite{LS}. Nonvariational cavitation problem has been recently studied in \cite{RT}. The delicate intermediary case, $0 < \gamma < 1$, addressed in this present work brings major novelty adversities as the equation satisfied within the positive set $\{u > 0 \}$ is nonhomogeneous and blows-up along the \textit{a priori} unknown quenching interface $\mathfrak{F} = \partial \{u > 0 \} \cap \Omega$ -  the so called free boundary of the problem. The lack of variational or energy approaches too implies significant difficulties in the problem and new, nonvariational solutions have to be established. In fact, since the free boundary problem considered in this paper has nonvariational character, one cannot use the powerful measure-distributional language to setup weak version of the problem. Instead we shall employ a perturbation scheme and will obtain uniform estimates with respect to the approximating parameter $\varepsilon$. A solution to the fully nonlinear free boundary problem \eqref{Damiao_Eduardo Eq02} will therefore be obtained as the limit of appropriate approximating configurations.
 
The first main problem to be addressed concerns the optimal regularity for solutions to Equation \eqref{Damiao_Eduardo Eq01}. Optimal estimates for heterogeneous equations, $\text{L}u = f(X,u)$ is in general a quite delicate issue. For the singular setting studied in this present work, optimal estimates are even more involved as they can be understood as invariant (tangential) equations for their own scaling. We show in Section \ref{Section sharp reg} of the present work that solutions  are locally of class $C^{1,\frac{\gamma}{2-\gamma}}$ at the free boundary. This result was only known in the variational setting, for minimizers of Euler-Lagrange functional, see \cite{P01, P02, ap} and \cite{GG, GG01}. 

The next principal result devilered in this article states that minimal solutions, i.e., solutions obtained from Perron's type method do grow precisely as $\text{dist}(X, \mathfrak{F})^{1+\frac{\gamma}{2-\gamma}}$, which corresponds to the maximum growth rate allowed. Such a result implies a quite restrictive geometry for the free quenching interface $\mathfrak{F}$.  As consequence of our sharp gradient estimate, Theorem \ref{regularitythm} and optimal growth rate, Theorem \ref{lemma Growth}, a minimal solution is trapped between the graph of two multiples of  $\text{dist}(X, \mathfrak{F})^{1+\frac{\gamma}{2-\gamma}}$, i.e., 
$$
	\uline{c} \cdot \text{dist}(X, \mathfrak{F})^{1+\frac{\gamma}{2-\gamma}} \le u(X) \le \overline{C} \cdot \text{dist}(X,\mathfrak{F})^{1+\frac{\gamma}{2-\gamma}}, \quad X \in \{u>0 \}.
$$ 
By means of geometric considerations, in Section \ref{Section Harnack} we establish a \textit{clean} Harnack inequality for solutions to \eqref{Damiao_Eduardo Eq01} within free boundary tangential balls, $B \subset \{u> 0 \}$, $B$ tangent to $\mathfrak{F}$.  In Section \ref{Section Hausdorff}, under an extra asymptotic structural assumption on the governing operator $F$, we establish Hausdorff estimates of the free boundary. In particular we show $\chi_{\{u > 0 \} \cap \Omega'} \in \text{BV}(\Omega)$, that is, $\{u> 0\}$ is locally a set of finite perimeter. We further show that the reduced free boundary has $\mathcal{H}^{n-1}$ total measure. The last two Sections close up the project by obtaining a solution to the fully nonlinear free boundary problem \eqref{Damiao_Eduardo Eq02} with the desired analytic and geometric properties.

\section{Mathematical set-up} \label{Section Math set-up}

Throughout this paper $\Omega$ will be a fixed Lipschitz bounded domain in $\mathbb{R}^N$, $f\colon \partial \Omega \to \mathbb{R}_{+}$ is a continuous boundary datum and $0 < \gamma < 1$ is a fixed real number. We shall denote by  $\mbox{Sym}(N)$ the space of all real $N\times N$ symmetric matrices and $F$ will be a uniformly elliptic fully nonlinear operator; that is, we shall assume  that there exist two constants $0<\lambda\leq \Lambda$ such that
\begin{equation}\label{uniform ellip}
	F(\mathcal{M}+\mathcal{N})\leq F(\mathcal{M})+\Lambda\|\mathcal{N}^{+}\|-\lambda\|\mathcal{N}^{-}\|, \quad  \forall \mathcal{M},\mathcal{N} \in \mbox{Sym}(N). 
\end{equation}

The ultimate goal of this paper is to study existence and fine qualitative properties of solutions to the singular equation
\begin{equation} \label{eq math setup}
	F(D^2u) = \gamma u^{\gamma -1} \cdot \chi_{\{u > 0 \}}.
\end{equation}
From the equation itself, one notices that the Hessian of an existing solution blows-up along the free boundary $\mathfrak{F} = \partial \{u>0\} \cap \Omega$; therefore, solutions cannot be of class $C^2$. In the fully nonlinear setting, the problem of optimal regularity for solutions to Equation \eqref{eq math setup} is a rather   delicate issue and it will be addressed in Section \ref{Section sharp reg}. Part of the subtleness of this problem comes from the intrinsic complexity of  the regularity theory for viscosity solutions to uniform elliptic equations. We recall that it is well known that solutions to homogeneous equation
\begin{equation} \label{eq hom math setup}
	F(D^2u) = 0,
\end{equation}
has a priori $C^{1,\mu}$ bounds for some $ \mu > 0$ that depends only on $N, \lambda$ and $\Lambda$. Under concavity or convexity assumption on $F$, a Theorem due to Evans and Krylov, states that solutions are $C^{2,\alpha}$. Nevertheless, Nadirashvili and Vladut have recently shown that given any $0<\eta < 1$ it is possible to build up a uniformly elliptic operator $F$, whose solutions to the homogeneous equation \eqref{eq hom math setup} are not $C^{1,\eta}$, see \cite{NV}, Theorem 1.1.  

%Therefore, in order to access the optimal regularity estimate available for the free boundary problem \eqref{eq math setup}, we shall assume  that $F$ has a priori $C^{2,\tau}$ estimates for %some small $0< \tau < 1$.  More precisely, we shall assume there exists $C_F > 0$ such that
%\begin{equation}\label{C2 est}
%	F(D^2h) = 0 \text{ in } B_1 \quad \text{implies } \quad \|h\|_{C^{2,{\tau}}(B_{1/2})} \le C_F \cdot \|h\|_{L^\infty(B_1)}.
%\end{equation}
%Hypothesis \eqref{C2 est} will be enforced hereafter in the paper, though all but Theorem \ref{regularitythm} do not depend on such condition (see also Remark \ref{rmk1}).

Let us turn our attention to the singularly perturbed strategy we shall use in order to grapple with the lack of variational approaches available. In this paper we suggest the following  singular perturbation scheme to appropriately approach the   free boundary problem \eqref{Damiao_Eduardo Eq02}:
\begin{equation} \label{Ee}  \tag{$E_\varepsilon$}
\left\{
\begin{array}{rllll}
    F(D^2u)  &=& \beta_{\varepsilon}(u), & \mbox{in} & \Omega \\
    u&=&f & \mbox{on} & \partial\Omega.  \\
\end{array}
\right.
\end{equation}
The singular perturbation term $\beta_\varepsilon$ is build up as follows: initially select your favorite  function $\rho\in C^{\infty}_0[0,1]$ and set
\begin{equation}\label{alpha}
	\alpha:=1+\dfrac{\gamma}{2-\gamma}.
\end{equation}
Throughout the whole paper, $\alpha$  will always be the fixed value stated in \eqref{alpha}. In the sequel, define
\begin{equation} \label{def B}
	 B_{\varepsilon}(t)=\int_{0}^{\frac{t-\varepsilon^{\alpha}\sigma_0}{\varepsilon^{\alpha}}}\rho(s)ds,
\end{equation}
where $0<\sigma_0<\frac{1}{2}$ is an arbitrary  technical choice. Notice that $B_\varepsilon$ is a smooth approximation of $\chi_{(0,\infty)}$. Finally, we set
\begin{equation} \label{beta}
	\beta_{\varepsilon}(t)=\gamma t^{\gamma-1}B_{\varepsilon}(t).
\end{equation}
Such a construction is carefully carried out as to preserve the natural scaling of the desired equation \eqref{eq math setup}.

%\subsection{Notation}

We finish this Section by listing the main notations adopted throughout the article:
\begin{itemize}
    \item The dimension of the Euclidean space the problem is modeled in will be denoted by $N \ge 2$. $\Omega$ will be a fixed bounded domain in $\mathbb{R}^N$. For a domain $\mathcal{O} \subset \mathbb{R}^N$, $\partial \mathcal{O}$ will represent the boundary of the domain $\mathcal{O}$. $\chi_{S}$ will stand for the characteristic function of the set $S$.
    \item The $N$-dimensional Lebesgue measure of a set $A\subset \mathbb{R}^N$ will be denoted by $\Leb(A)$. $\mathcal{H}^{n-1}$ will stand for the $(n-1)$-Hausdorff measure.
    \item $\left\langle\cdot ,\cdot \right\rangle$ will be the standard scalar product in $\mathbb{R}^N$. For a vector $\xi=(\xi_1,\cdots,\xi_N) \in \mathbb{R}^N$, its Euclidean norm will be denoted by $|\xi| := \sqrt{\langle \xi, \xi \rangle}$. The tensor product $\xi\otimes \psi$ denotes the matrix whose entries are given by $\xi_i\psi_j$ for $1\leq i,j\leq N$. 
    \item $B_{r}(p)$ will be the open ball centered at $p$ with radius $r$. Furthermore, we shall denote $kB = kB_{r}(p):= B_{kr}(p)$, for any $k>0$.
	\item $\text{Spect}(\mathcal{M})$ denotes the set of eigenvalues of the matrix $\mathcal{M}$.
	\item Fixed the ellipticity constants $0<\lambda \le \Lambda$, $\mathscr{P}^{-}(M)$ and $\mathscr{P}^{+}(M)$ denote the Pucci extremal operators:
	$$
		\begin{array}{lll}
			\mathscr{P}^{-}(M) &:=& \inf \left \{ \text{Tr} (A\cdot M) \suchthat \lambda \text{Id} \le A \le \Lambda \text{Id}  \right \}, \\
			\mathscr{P}^{+}(M) &:=& \sup \left \{ \text{Tr} (A\cdot M) \suchthat \lambda \text{Id} \le A \le \Lambda \text{Id}  \right \}.
		\end{array}
	$$
    \item Constants $C, C_1, C_2, \cdots >0$ and $c, c_0, c_1, c_2, \cdots > 0$ that depend only on dimension, $\gamma$ and ellipticity constants $\lambda, ~ \Lambda$ will be call universal. Any additional dependence will be emphasized.
\end{itemize}

\section{Existence of minimal solutions} \label{Section existence}

In this section we comment on the existence of a viscosity solution to  equation \eqref{Ee}. More importantly, we shall establish herein a stable process to select special solutions to \eqref{Ee}. As we will show in Section \ref{Section Nondeg}, the family of minimal solutions turns out to satisfy the  desired appropriate geometric features. Such properties will allow us to establish Hausdorff estimates of the free boundary in Section \ref{Section Hausdorff}.

Notice that because of the lack of monotonicity of equation \eqref{Ee} with respect to the variable $u$, classical Perron's method cannot be directly employed. The next theorem proved in \cite{RT}, is an adaptation of Perron's method, which is by now fairly well understood.

\begin{theorem}\label{existthm}
Let $g$ be a bounded, Lipschitz function defined in the real line $\mathbb{R}$. Suppose $F$ uniformly elliptic and that the equation $F (D^2 u) = g(u)$ admits a Lipschitz viscosity subsolution $u_\star$ and a Lipschitz viscosity supersolution $u^\star$ such that $u_\star=u^\star=f\in C(\partial\Omega)$. Define the set of functions,
\[
S:= \left \{ w\in C(\overline{\Omega}); \quad u_\star\leq w \leq u^\star \quad \mbox{and}\; w \; \mbox{supersolution of}\; F(D^2u)=g(u) \right \}.
\]
Then,
\[
v(x):=\inf_{w\in S} w(x)
\]
is a continuous viscosity solution of $F(D^2u)=g(u)$ and $v=f$ continuously on $\partial\Omega$.
\end{theorem}

Existence of minimal solution to Equation  \eqref{Ee} follows by choosing
$u_\star=u_\star(\varepsilon)$ and $u^ \star=u^ \star(\varepsilon)$ solutions to the following boundary value problems
\[
\begin{array}{lcr}
\left\{
\begin{array}{rllll}
    F(D^2u_\star)  &=& \zeta, & \mbox{in} & \Omega \\
    u_\star&=&f & \mbox{on} & \partial\Omega,  \\
\end{array}
\right.

&\mbox{and}&

\left\{
\begin{array}{rllll}
    F(D^2u^\star)  &=& 0, & \mbox{in} & \Omega \\
    u^\star&=&f & \mbox{on} & \partial\Omega,  \\
\end{array}
\right.
\\
\end{array}
\]
where 
$$
	\zeta:=\sup\beta_\varepsilon \sim \varepsilon^{\gamma -1}.
$$ 
The existence the functions $u_\star$ and $u^\star$ is consequence of standard Perron's method. By construction $u_\star$ is viscosity subsolution of $(E_\varepsilon)$ and $u^\star$ is a viscosity supersolution of \eqref{Ee}. Note that $u^\star,u_\star \in C^{0,1}(\Omega) \cap C(\overline{\Omega})$. Thus a direct application of Theorem $\ref{existthm}$ yields the following existence result:

\begin{theorem}[Existence of minimal solutions]\label{minexist}
Let $\Omega\in\mathbb{R}^n$ be a Lipschitz domain and $f \in C(\partial\Omega)$ be a nonnegative boundary datum. Then, for each $\varepsilon>0$ fixed, equation \eqref{Ee} has a nonnegative minimal viscosity solution $u_{\varepsilon}\in C(\bar{\Omega})$.
\end{theorem}

As previously mentioned, more importantly than assuring existence of a viscosity solution to \eqref{Ee}, Theorem $\ref{minexist}$ provides a particular choice of solutions to such an equation. In comparison with the variational theory, this choice is a replacement for
the selection of minimizers of the Euler-Lagrange functional (see for instance \cite{T} for further details). Therefore, unless otherwise stated, whenever we mention viscosity solution to \eqref{Ee}, we mean the minimal solution provided by Theorem $\ref{minexist}$.

%%%%%%%%%%%%%%%%
\section{Sharp regularity estimates} \label{Section sharp reg}

The first main result we prove in this paper is the optimal regularity estimate, uniform in $\varepsilon$, available for solutions to \eqref{Ee}. We will show that $u_\varepsilon$ is locally a $C^{1,\beta}$ function and we shall further determine the optimal $\beta>0$ in terms of the degree of singularity $\gamma$. This key information has only been known for variational solutions, \cite{P01, GG, GG01} and the proofs make decisive use of energy considerations. In principle it is not even clear that one should expect the same regularity theory for nonvariational problems. 

Thus, we start off this Section by rather informal, heuristic considerations as to guide us through the genuine results to be established later on. Let us analyze the limiting free boundary problem \eqref{Damiao_Eduardo Eq02}. Suppose $0$ is a free boundary point and, say, $-\text{e}_n$ is the unit outward normal pointing towards the quenching phase $\{u = 0\}$. If  $u$ is $C^{1,\beta}$ at $0$, then, in a small neighborhood, say, $B_\rho \cap \{u > 0 \}$, $\rho \ll 1$, $u$ behaves like $\sim X_n^{1+\beta}$. Therefore, the singular potential of the equation in \eqref{Damiao_Eduardo Eq02} is like $\sim X_n^{(1+\beta) \cdot (1-\gamma)}$.  In view of the regularity theory for heterogeneous fully nonlinear equations $F(D^2u) = f(X)$, established in \cite{C2} and \cite{T2}, we obtain the following implication
$$
	X_n^{(1+\beta) \cdot (1-\gamma)} \in L^\theta_\text{weak} \quad \text{implies} \quad u \in C^{1, 1-\frac{1}{\theta}}.
$$
The reasoning above gives the following system of algebraic equations
$$
	\left \{ 
		\begin{array}{rll}
			\theta(1+\beta)(\gamma - 1) &=& -1 \\
			\beta &=& 1-\frac{1}{\theta}.
		\end{array}
	\right.
$$
Solving for $\beta$, revels, $\beta = \frac{\gamma}{2-\gamma}$, which agrees with the optimal regularity estimate established for the variational theory.  

\par

This Section is devoted to establish local $C^{1,\frac{\gamma}{2-\gamma}}$ regularity estimates for  solutions $u_\varepsilon$ to Equation \eqref{Ee}, uniform in $\varepsilon$ at free boundary points. %Recall that we are working under the natural assumption that $F$ has \textit{a priori} $C^{2,\tau}$ estimates. 
In fact we shall obtain a universal growth control on $u_\varepsilon$ near the free boundary. The desired regularity along the free interface will then follow.
\par

Hereafter, let us fix a point $X_0 \in \Omega$ and for simplicity take $X_0 = 0$. Our analysis will be based on the auxiliary function $  v_\varepsilon$, defined by
\begin{equation}\label{thm4.1 eq00}
	v_\varepsilon(X) := u_\varepsilon^{\frac{2-\gamma}{2}}(X).
\end{equation}
For the sake of notation convenience, let us omit the subscript $\varepsilon$ in $v_\varepsilon$ and in $u_\varepsilon$, writing simply $v$ and $u$ to denote these functions. 
Formally one computes
\begin{eqnarray}	
		\nabla v(X) &=& \left (1-\frac{\gamma}{2}\right ) u^{\frac{-\gamma}{2}}(X) \cdot \nabla u(X) \label{thm4.1 eq01}\\
		D^2 v(X) &=&   -\frac{\gamma}{2} \left (1-\frac{\gamma}{2}\right ) u^{-1 - \frac{\gamma}{2}}(X) \cdot \nabla u(X) \otimes \nabla u(X) +  \left (1-\frac{\gamma}{2}\right ) u^{\frac{-\gamma}{2}}(X) \cdot D^2 u(X) \label{thm4.1 eq02}
\end{eqnarray}
Plugging \eqref{thm4.1 eq01} into \eqref{thm4.1 eq02} yields
\begin{equation}\label{thm4.1 eq03}
	D^2v(X) = -\left ( \frac{\gamma}{2-\gamma} \right ) v^{-1}(X) \cdot \nabla v(X) \otimes \nabla v(X) + \left (1-\frac{\gamma}{2}\right ) u^{1-\gamma}(X) \cdot D^2 u(X) \cdot v^{-1}(X).
\end{equation}
From the PDE satisfied by $u$, we have
\begin{eqnarray}
	u^{1-\gamma}(X)  v^{-1}(X) \cdot F(D^2u) &=& \gamma \left ( u^{1-\gamma} \cdot \beta_\varepsilon(u) \right ) \cdot v^{-1}(X) \label{thm4.1 eq04} \\
	&= & \gamma B_\varepsilon(v^\alpha) \cdot v^{-1}(X) \label{thm4.1 eq05} 
\end{eqnarray}
Thus, ellipticity and \eqref{thm4.1 eq03} yield
\begin{equation}\label{thm4.1 eq7}
	F \left ( D^2v(X) +  \frac{\gamma}{2-\gamma}  v^{-1}(X) \cdot \nabla v(X) \otimes \nabla v(X) \right ) = f(X) v^{-1}(X),
\end{equation}
for a bounded function $f(X)$. Though the above computation has been conducted formally, it is standard to justify Equation \eqref{thm4.1 eq7} using the language of viscosity solutions. 

\par

Our first result towards optimal regularity establishes equicontinuity for functions satisfying \eqref{thm4.1 eq7}, which implies the same conclusion to the family of functions $u_\varepsilon$. The proof is an adaptation of the Ishii-Lions method \cite{IL}, see also \cite{B},  \cite{BCI} and \cite{IS}.

%%%%%%%%%%%
%%%%%%%%%%%

\begin{proposition}[$C^0$-compactness] \label{equicontinuity} Solutions to \eqref{thm4.1 eq7} are universally locally uniform continuous, that is, there exists a universal modulus of continuity, $\varrho$, such that $|v(X_1) - v(X_2)| \le C_{\Omega'} \varrho(|X_1-X_2))$, for $X_1, X_2 \in \Omega' \Subset \Omega$.
\end{proposition}

\begin{proof}
	Fixed $X_0 \in \Omega$, let us denote by $r = \dist(X_0, \partial \Omega)$. We will show that for any $\delta>0$ given, there exists $L_\delta > 0$ such that
	\begin{equation}\label{propeq0}
		\Phi:= \sup\limits_{(X,Y)\in \bar{\Omega} \times \bar{\Omega}} \left \{v(X) - v(Y) - L_\delta \omega (|X-Y|) - \dfrac{8\|v\|_\infty}{r^2} (|X-X_0|^2 + |Y-X_0|^2 )\right \} \le \delta,
	\end{equation}
	where $\omega(t) = t - \frac{1}{10\sqrt{r}} t^{3/2}$, for $t\le r$, $\omega(t) =  \frac{9}{10}r$ for $t\ge r$.  For that, suppose $\Phi > \delta$ and let $(\bar{X}, \bar{Y})$ be a maximum point. It readily follows that
	\begin{equation}\label{propeq1}
		  \dfrac{8\|v\|_\infty}{r^2} (|\bar{X}-X_0|^2 + |\bar{Y}-X_0|^2 ) < 2 \|v\|_\infty,
	\end{equation}
 	thus, $\bar{X}$ and $\bar{Y}$ are interior points and $|\bar{X}- \bar{Y}| < r$. Clearly $\bar{X} \not = \bar{Y}$, otherwise $\Phi = 0 < \delta$. Define in the sequel the vectors
	\begin{eqnarray} 
		\xi_X := L_\delta \omega'(|\bar{X} -\bar{Y}|)\eta +  \dfrac{16\|v\|_\infty}{r^2} (\bar{X} - X_0)	\label{propeq2}\\
		\xi_Y := L_\delta \omega'(|\bar{X} -\bar{Y}|)\eta -  \dfrac{16\|v\|_\infty}{r^2} (\bar{Y} - X_0)	\label{propeq3},
	\end{eqnarray}
	where $\eta := \frac{\bar{X} -\bar{Y}}{|\bar{X} -\bar{Y}|}$. From Jensen-Ishii's approximation Lemma, see \cite{UserG} and also \cite{BCI}, for $\iota > 0$ small enough, it is possible to find matrices $M_X$ and $M_Y$ with 
	\begin{eqnarray}
		(\xi_X, M_X)& \in &{J}^{-}(v, \bar{X}), \label{propeq2.1}\\
		(\xi_Y, M_Y) &  \in & {J}^{+}(v, \bar{Y}),  \label{propeq2.2}
	\end{eqnarray} 
where ${J}^{-}$ and ${J}^{+}$ denote the subjet and superjet respectively (see \cite{UserG} for definition), verifying the following matrix inequality
	\begin{equation}\label{propeq4}
		\begin{pmatrix}
M_X & 0 \\
0 & -M_Y 
\end{pmatrix}
\leq 
\begin{pmatrix}
Z  & -Z \\
-Z & Z
\end{pmatrix} + (\dfrac{16\|v\|_\infty}{r^2}+\iota) \text{Id}_{2n\times 2n} , 	
\end{equation}	
where 
\begin{equation}\label{propeq5}
		Z = \omega''(|\bar{X} -\bar{Y}|) \dfrac{ (\bar{X} -\bar{Y}) \otimes (\bar{X} -\bar{Y}) }{|\bar{X} -\bar{Y}|^2} + \dfrac{\omega'(|\bar{X} -\bar{Y}|)}{|\bar{X} -\bar{Y}|} \left \{ \text{Id}_{n\times n}  - \dfrac{ (\bar{X} -\bar{Y}) \otimes (\bar{X} -\bar{Y}) }{|\bar{X} -\bar{Y}|^2}  \right \}.
\end{equation}
Applying inequality \eqref{propeq4} to vectors of the form $(\xi, \xi)$, we conclude 
\begin{equation}\label{propeq6}
		\text{Spect}(M_Y - M_X) \in (-\dfrac{32\|v\|_\infty}{r^2} - \iota, +\infty).
\end{equation}
However, if we apply to the special vector $(\eta, -\eta)$, we conclude 
\begin{equation}\label{propeq7}
		\text{Spect}(M_Y - M_X) \cap (\frac{c}{\sqrt{r}} L_\delta -\dfrac{32\|v\|_\infty}{r^2} - \iota  +\infty) \not = \emptyset,
\end{equation}
for a universal number $c>0$, easily computed if one desires. 

\par

In the sequel we use Equation \eqref{thm4.1 eq7}, together with \eqref{propeq2.1} and \eqref{propeq2.2} to write up the following ponitwise inequalities 
\begin{eqnarray}
	F \left ( M_X +  \frac{\gamma}{2-\gamma}  v^{-1}(\bar{X}) \cdot \xi_X \otimes \xi_X \right ) \ge  f(\bar{X}) v^{-1}(\bar{X}) \label{propeq8}\\
	F \left ( M_Y +  \frac{\gamma}{2-\gamma}  v^{-1}(\bar{Y}) \cdot \xi_Y \otimes \xi_Y \right ) \le  f(\bar{Y}) v^{-1}(\bar{Y}) \label{propeq9}.
\end{eqnarray}
Subtracting \eqref{propeq8} from \eqref{propeq9} and using ellipticity, we find
\begin{equation}\label{propeq10}
	\mathscr{P}^{-} \left ( [M_Y - M_X]  +  \frac{\gamma}{2-\gamma}  [v^{-1}(\bar{Y}) \cdot \xi_Y \otimes \xi_Y - v^{-1}(\bar{X}) \cdot \xi_X \otimes \xi_X] \right ) \le f(\bar{Y}) v^{-1}(\bar{Y}) -  f(\bar{X}) v^{-1}(\bar{X}),
\end{equation}
where $\mathscr{P}^{-}$ is the extremal Pucci operator with ellipticity $(\lambda, \Lambda)$. If we label $\Xi :=  [M_Y - M_X]  +  \frac{\gamma}{2-\gamma}  [v^{-1}(\bar{Y}) \cdot \xi_Y \otimes \xi_Y - v^{-1}(\bar{X}) \cdot \xi_X \otimes \xi_X] $, from the definition of $\mathscr{P}^{-}$, there exists a $(\lambda, \Lambda)$-elliptic matrix, $\lambda \text{Id} \le a_{ij} \le \Lambda \text{Id}$, satisfying 
\begin{equation}\label{propeq11}
	\text{Tr}(a_{ij} \Xi_{ij})  - 1 \le \mathscr{P}^{-}(\Xi).
\end{equation}
From \eqref{propeq6}, \eqref{propeq7} and ellipticity, we estimate
\begin{equation}\label{propeq12}
	\text{Tr} \left (a_{ij} (M_Y - M_X)_{ij}  \right )  \ge \lambda \frac{c}{\sqrt{r}} L_\delta - n\Lambda\left ( \frac{32\|v\|_\infty}{r^2} - \iota \right ).
\end{equation} 
Now we compute
\begin{equation}\label{propeq13}
	\text{Tr} \left (a_{ij} ( \xi_Y \otimes \xi_Y )_{ij} \right )  = \left ( \frac{17}{20}L_\delta \right )^2 \cdot \text{Tr} \left (a_{ij} \eta_i\eta_j \right ) -  \frac{256 \|v\|^2_\infty}{r^4} \cdot \text{Tr} \left (a_{ij}(\bar{Y}-X_0)_i (\bar{Y}-X_0)_j \right ),
\end{equation} 
and likewise
\begin{equation}\label{propeq14}
	\text{Tr} \left (a_{ij} ( \xi_X \otimes \xi_X )_{ij} \right )  = \left ( \frac{17}{20}L_\delta \right )^2 \cdot \text{Tr} \left (a_{ij} \eta_i\eta_j \right ) +  \frac{256 \|v\|^2_\infty}{r^4} \cdot \text{Tr} \left (a_{ij}(\bar{X}-X_0)_i (\bar{X}-X_0)_j \right ). 
\end{equation} 
From ellipticity follows the estimates 
\begin{eqnarray}\label{propeq15}
	\sigma:= \text{Tr} \left (a_{ij} \eta_i\eta_j \right ) &\ge& \lambda, \label{propeq15} \\
	\max \left \{ \text{Tr} \left (a_{ij}(\bar{Y}-X_0)_i (\bar{Y}-X_0)_j \right ), \text{Tr} \left ( a_{ij}(\bar{X}-X_0)_i (\bar{X}-X_0)_j \right ) \right \} & \le &\Lambda r^2. \label{propeq16}
\end{eqnarray} 
Also, since $v(\bar{X}) - v(\bar{Y}) > \delta$, we readily have
\begin{eqnarray} \label{propeq16.1}
	v^{-1}(\bar{X}) & < &  \dfrac{1}{\delta}  \label{propeq16.1} \\ 
	v^{-1}(\bar{Y}) & > & \frac{\delta}{\|v\|_\infty^2} + v^{-1}(\bar{X}).  \label{propeq16.2} 
\end{eqnarray}
Combining \eqref{propeq10},  \eqref{propeq11},  \eqref{propeq12},  \eqref{propeq13} and \eqref{propeq14}, if $L_\delta \gg 1$, depending only on universal numbers, and $r$, we obtain, 
\begin{equation} \label{propeq17}
	\begin{array}{lll}
		0 & < & \lambda \frac{c}{\sqrt{r}} L_\delta - n\Lambda\left ( \frac{32\|v\|_\infty}{r^2} - \iota \right ) -1  \\
		& \le &  \left [ \frac{\gamma}{2-\gamma}   \left (\frac{17}{20}L_\delta \right )^2 \cdot \sigma +  \frac{\gamma}{2-\gamma} \cdot   \frac{256\Lambda  \|v\|^2_\infty}{r^2} - f(\bar{X})  \right ] v^{-1}(\bar{X}) \smallskip \\
		&- & \left [ \frac{\gamma}{2-\gamma}   \left (\frac{17}{20}L_\delta \right )^2 \cdot \sigma -  \frac{\gamma}{2-\gamma} \cdot   \frac{256\Lambda  \|v\|^2_\infty}{r^2} - f(\bar{Y})  \right ] v^{-1}(\bar{Y}).
	\end{array}
\end{equation}
Now, we select $L_\delta$ even bigger, depending further on $\gamma$, $\|v\|_\infty$, $\|f\|_\infty$ so that
\begin{equation} \label{propeq18}
	\dfrac{\frac{\gamma}{2-\gamma}   \left (\frac{17}{20}L_\delta \right )^2 \cdot \sigma +  \frac{\gamma}{2-\gamma} \cdot   \frac{256\Lambda  \|v\|^2_\infty}{r^2} - f(\bar{X})}{ \frac{\gamma}{2-\gamma}   \left (\frac{17}{20}L_\delta \right )^2 \cdot \sigma -  \frac{\gamma}{2-\gamma} \cdot   \frac{256\Lambda  \|v\|^2_\infty}{r^2} - f(\bar{Y})} < \left ( 1+ \dfrac{\delta^2}{\|v\|_\infty^2} \right ).
\end{equation}
Estimates \eqref{propeq16.2}, \eqref{propeq17} and \eqref{propeq18} give
\begin{equation} \label{propeq18.1}
	\dfrac{\delta}{\|v\|^2_\infty} + v^{-1}(\bar{X}) < v^{-1}(\bar{Y}) < \left ( 1+ \dfrac{\delta^2}{\|v\|_\infty^2} \right ) v^{-1}(\bar{X}).
\end{equation}
Finally, confronting \eqref{propeq18.1} and \eqref{propeq16.1} we end up with
\begin{equation} \label{propeq18.2}
	\dfrac{\delta}{\|v\|^2_\infty}  <  \dfrac{\delta^2}{\|v\|_\infty^2}  v^{-1}(\bar{X}) < \dfrac{\delta}{\|v\|^2_\infty},
\end{equation}
which is a contradiction. Thus, if $L_\delta \gg 1$, indeed $\Phi < \delta$ and the proof of Proposition \ref{equicontinuity} follows. 
\end{proof}

%%%%%%%%%
%%%%%%%%%
%%%%%%%%%

An inference from the proof above reveals that in fact $L_\delta \sim \delta^{-1}$, as to attain \eqref{propeq18}. Thus, Proposition \ref{equicontinuity} gives local $C^{0,\frac{1}{2}}$ continuity for $v$. We further comment that the proof could also be performed using $\omega(t) = t^\theta$, with $0< \theta < 1$, as $|X-Y|^\theta$ too is concave in the radial direction (information used to obtain \eqref{propeq7}). We shall use this observation later. 

%%%%%%%%%
%%%%%%%%%
%%%%%%%%%

\begin{theorem}[Uniform optimal regularity]\label{regularitythm}
Given a subset $\Omega'\Subset \Omega$, there exists a constant $C$ depending on, $\|f\|_{\infty}$, $\gamma$, $\Omega'$, dimension, ellipticity, but independent of $\varepsilon$, such that, any family of viscosity solutions $\{u_{\varepsilon}\}$ of equation \eqref{Ee} satisfies,
\begin{equation}\label{gradest1}
 \sup\limits_{B_r(X)} u_\varepsilon \le C \left ( r^{\frac{2}{2-\gamma}} + u_\varepsilon(X) \right ), \quad \forall X \in \Omega'.
\end{equation}
\end{theorem}
%%%%
\begin{proof} 
Suppose, for the sake of contradiction, the thesis of Theorem \ref{regularitythm} fails to hold. Combining discrete iterative techniques and a continuous methods, see \cite{CKS}, Lemma 3.3 and also \cite{ST} for similar reasoning, for each $k>1$, it is possible to find $0<r_k = \text{o}(1)$, $X_k \in \Omega'$, $\varepsilon_k > 0 $  such that the following two inequalities hold
\begin{eqnarray}
	s_k:= \sup\limits_{B_{r_k}(X_k)} u_{\varepsilon_k}  &>& k \left (r_k^{\alpha} +  u_{\varepsilon_k} (X_k) \right ) \label{regeq03} \\
	\sup\limits_{B_{r_k}(X_k)} \left [ u_{\varepsilon_k}- u_{\varepsilon_k}(X_k) \right ] &\ge& 2^{-\alpha k}  \sup\limits_{B_{2^k r_k}(X_k)} \left [ u_{\varepsilon_k}- u_{\varepsilon_k}(X_k) \right ]. \label{regeq03.1}
\end{eqnarray}
The normalized function $\varphi_k \colon B_1 \to \mathbb{R}$ given by
\begin{equation}\label{regeq05}
	\varphi_k(Y) := \dfrac{u_{\varepsilon_k}(X_k + r_kY)}{s_{k}},
\end{equation} 
satisfies
\begin{eqnarray}
	 0 \le &\varphi_k(Y)&  \le C |Y|^\alpha  \ , \label{regeq06} \\
	& \varphi_k(0)& = \text{o}(1), \label{regeq07} \\ 
	 &\sup\limits_{B_{1}} \varphi_k &= 1.  \label{regeq08}
\end{eqnarray}
In addition, the following equation is satisfied in the viscosity sense
\begin{equation}\label{regeq09}
	{F}_k(D^2 \varphi_k(Y)) = \left ( \dfrac{r^{\alpha}_k}{s_k} \right )^{\frac{2}{\alpha}} \beta_{\varepsilon_k}(\varphi_k),
\end{equation} 
where 
\begin{equation}\label{regeq09.1}
	{F}_k(\mathcal{M}) := \dfrac{r_k^2}{s_{k}} \cdot  F \left (\dfrac{s_k}{{r_k^2}} \mathcal{M} \right ),
\end{equation} 
which is a $(\lambda, \Lambda)$-elliptic operator. Thus, from Proposition \ref{equicontinuity}, up to a subsequence, $\varphi_k$ converges locally uniformly to an entire function $\varphi_0 \colon \mathbb{R}^N \to \mathbb{R}$. From hypothesis of contradiction \eqref{regeq03},
\begin{equation}\label{regeq11}
	\left ( \dfrac{r^{\alpha}_k}{s_{k}}\right )^{\frac{2}{\alpha}}  = \text{o}(1).
\end{equation} 
Passing to another subsequence, if necessary, $F_k$ converges locally uniformly to a limiting \textit{recession} operator $\tilde{F}$, which is $(\lambda, \Lambda)$-elliptic and homogeneous of degree one for nonnegative scalars. Passing the limit as $k \to \infty$ in \eqref{regeq09} yields
\begin{equation}\label{regeq12}
	\varphi_0^{1-\gamma} \cdot \tilde{F}(D^2 \varphi_0(Y)) = 0.
\end{equation} 
Notice, in view of \eqref{regeq07} and \eqref{regeq08}, we have
\begin{equation}\label{regeq10}
	\varphi_0(0) = 0, \quad \sup\limits_{B_{1}} \varphi_0 = 1.
\end{equation} 
We now revisit the proof of Proposition \ref{equicontinuity}. Defining $\psi := \varphi_0^{1/\alpha}$, we find 
\begin{equation}\label{regeq13}
	\tilde{F}( D^2 \psi + c_\gamma \psi^{-1} \nabla \psi \otimes \nabla \psi) = 0.
\end{equation} 
By running the same reasonings of the proof of Proposition \ref{equicontinuity}, for $\psi$, with $\omega = t^\theta$, $0<\theta<1$, $\delta =0$, $f=0$ and with no localization term, gives $C^{0,\theta}$ estimates for $\psi$, for any $\theta <1$. In fact,  for $L \gg 1$, depending only on ellipticity, estimate \eqref{propeq17} becomes
\begin{equation}\label{regeq14}
	0 < cL^2 \cdot \left ( \psi^{-1}(\bar{X}) - \psi^{-1}(\bar{Y}) \right ) < 0,
\end{equation} 
since the contradiction assumption in the reasoning of the proof of Proposition \ref{equicontinuity} implies $\psi(\bar{X}) > \psi(\bar{Y})$. We now choose $\theta_0 $ so that
\begin{equation}\label{regeq14.1}
	\frac{1}{\alpha} < \theta_0 < 1.
\end{equation}
A final contradiction is then obtained when we confront \eqref{regeq10} with the $C^{0,\theta_0}$ regularity for $\psi$. Indeed, select a point $Z$ in $\{\varphi_0 > 0\}$ and $Z_0 \in \{\varphi_0 = 0\}$, satisfying $\dist (Z, \{\varphi_0 = 0\}) =   |Z-Z_0| < \frac{1}{2}.$ 
It follows from Hopf maximum principle that 
\begin{equation}\label{regeq14.2}
	0< \liminf\limits_{h\to 0} \frac{\varphi(h e + Z_0)}{h},
\end{equation}
where $e$ is the inward normal vector to the ball $B_{|Z-Z_0|}(Z)$ at $Z_0$. On the other hand, we have
\begin{equation}\label{regeq14.2}
	\begin{array}{lll}
		\dfrac{\varphi_0(he+Z_0)}{|h|} &=& \sqrt[\theta_0]{\dfrac{\varphi_0^{\frac{1}{\alpha}}(he+Z_0)}{|h|^{\theta_0}}} \cdot \varphi_0^{1-\frac{1}{\alpha\theta_0}}(he + Z_0) \\
		&\le& C \cdot \varphi_0^{\delta_0}(he + Z_0) \\
		&\to & 0,
	\end{array}
\end{equation}
as $h \to 0$. 
This concludes the proof of Theorem \ref{regularitythm}.

\end{proof}

\section{Nondegeneracy of minimal solutions} \label{Section Nondeg}

In the previous Section we have shown that solutions to Equation \eqref{Ee} are locally of class $C^{1, \frac{\gamma}{2-\gamma}}$. In particular such an estimate provides an upper bound on how fast $u_\varepsilon$ growths away from, say, the level surface $\{u_\varepsilon \sim \varepsilon^\alpha\}$, for $\alpha$ as in \eqref{alpha}. That is, 
$$
	u_\varepsilon(Z) \lesssim \left [ \text{dist}(Z, \{u_\varepsilon \sim \varepsilon^\alpha\})\right ]^\alpha.
$$ 
The main result we shall prove in this Section states that minimal solutions do growth precisely as $\text{dist}(X_0, \{u_\varepsilon \sim \varepsilon^\alpha\})^\alpha$, see Corollary \ref{Growth col} for the precise statement. In fact we shall establish a stronger nondegeneracy property of minimal solutions, which also has fundamental importance in our blow-up analysis.

To simplify the statement of the results, we introduce some definitions and notations. Hereafter we shall use systematically we following notations:
\[
\begin{array}{rcl}
\{u_\varepsilon>\kappa\} & := & \{x\in\Omega \mid u_\varepsilon(x)>\kappa\}, \\
\{\tau>u_\varepsilon>\lambda\} & := & \{x\in\Omega \mid \tau>u_\varepsilon(x)>\lambda\}, \\
d_\varepsilon(X) & := & \dist(X,\partial\{u_\varepsilon > \varepsilon^\alpha \}), \\
%B_\varepsilon(X) & := & B_{d_\varepsilon(X)}(X).\\ 
\end{array}
\]

The nondegeneracy feature of minimal solutions is based on the construction of appropriate viscosity supersolution whose value within an inter disk is much smaller than its value on the boundary of an outer disk. 
  
\begin{proposition} \label{prop barrier}
Assume, with no loss of generality that $0\in\Omega$. Given $0<\eta$, there exists a radially symmetric function $\theta\in C^{1,1}(\Omega)$ and universal small constants $0<c_2<1$ and $0<c_1<1$ such that
\begin{enumerate}
\item $\theta\equiv 2\sigma_0$ in $B_{c_1\eta}$
\item $\theta\geq c_2\eta^{1+\frac{\gamma}{2-\gamma}}$ in $\Omega \setminus B_\eta$
\item $\theta$ is satisfies $F(D^2\theta(X)) \le \beta(\theta(X)),$ pointwise in $\Omega$, where $\beta=\beta_1$, as in \eqref{beta}.
\end{enumerate}
\end{proposition}

\begin{proof}
Initially define
\[
\theta(X)=
\left\{
\begin{array}{lll}
2\sigma_0 & \mbox{for} & 0\leq|X|\leq c_1\eta; \\
a_0(|X|-c_1\eta)^2+2\sigma_0 & \mbox{for} & c_1\eta\leq |X|\leq \eta; \\
A|X|^{\alpha}+B & \mbox{for} & |X|\geq\eta. 
\end{array}
\right.
\]
where the constants $a_0,A,B$ e $c_1$ will be chosen later.  Our first goal is to enforce that such a function is indeed $C^{1,1}$. For this, we have to set along $|X|=\eta$, 
\begin{equation}
	a_0(1-c_1)^2\eta^2+2\sigma_0=\theta(X)=A\eta^{\alpha}+B \label{eq1} 
\end{equation} 
thus, easily we obtain
$$
	a_0=\frac{1}{(1-c_1)^2}\left[A\eta^{\alpha-2}+\eta^{-2}(B-2\sigma_0)\right].
$$ 
Moreover, differentiating  $\theta$ and matching its gradient along  $|X|=\eta$, we obtain
\begin{equation} \label{eq1.1} 
	2a_0(1-c_1)X_i=A\alpha\eta^{\alpha-2}X_i.
\end{equation} 
Combining \eqref{eq1} and \eqref{eq1.1} we find
\begin{equation}
\dfrac{A\alpha\eta^ {\alpha-2}}{2(1-c_1)}=\dfrac{1}{(1-c_1)^2}\left[A\eta^{\alpha-2}+\eta^{-2}\left(B-2\sigma_0\right)\right]. \label{eq2}
\end{equation}
In the sequel, take 
$$
	c_1:=\dfrac{\gamma}{2}\in (0,1),
$$
which implies the relation $\alpha=\frac{1}{1-c_1}$, where, as always, $\alpha$ is the value set in \eqref{alpha} . Finally we set
\[
	B=2\sigma_0-\dfrac{A}{2}\eta^\alpha,
\]
as to \eqref{eq2} to be  satisfied. Summing up the construction so far, we have built up 
\[
\theta(X)=
\left\{
\begin{array}{lll}
2\sigma_0 & \mbox{for} & 0\leq|X|\leq c_1\eta; \\
\dfrac{A\alpha^2}{2}\eta^{\alpha-2}(|X|-c_1\eta)^2+2\sigma_0 & \mbox{for} & c_1\eta\leq |X|\leq \eta; \\
A|X|^{\alpha}+\left(2\sigma_0-\dfrac{A}{2}\eta^{\alpha}\right) & \mbox{for} & |X|\geq\eta. 
\end{array}
\right.
\]
which is of $C^{1,1}$, by the construction itself. We still have the parameter $A$ to be adjusted later. Our next step is to show that $\theta$ is an appropriate supersolution, that is, we want to establish 
\begin{equation}
	F(D^2\theta)\leq \beta(\theta) \label{ineq1} 
\end{equation}
pointwise. To this end, we first analyze the equation in the region $c_1\eta\leq|X|\leq\eta$. Direct computations yield
\[
	\theta_{ij}=A\alpha^2\eta^{\alpha-2}\left[\dfrac{X_iX_j}{|X|^2}+\left(1-\dfrac{c_1\eta}{|X|}\right)\left(\delta_{ij}-\dfrac{X_iX_j}{|X|^2}\right)\right],
\]
within $c_1\eta\leq|X|\leq\eta$. At a point of the form $\bar{X}=(|X|,0,\cdots,0)$, we find out
\[
\begin{array}{rll}
\theta_{11}= & A\alpha^2\eta^{\alpha-2} &  \\
\theta_{ii}= & A\alpha^2\eta^{\alpha-2}\left(1-\dfrac{c_1\eta}{|X|}\right) & \mbox{if } i>1 \\  
\theta_{ij}= & 0 & \mbox{if } i\neq j. \\
\end{array}
\]
By symmetric invariance of $\theta$ and ellipticity of $F$, we obtain
\begin{equation}
F(D^2\theta(X))\leq\Lambda\left[A\alpha^2\eta^{\alpha-2}+(N-1)A\alpha^2\eta^{\alpha-2}\left(1-\dfrac{c_1\eta}{|X|}\right)\right]\leq\Lambda NA\alpha^2\eta^{\alpha-2}.\label{ineq2}
\end{equation}
Recall $N$ is the dimension of the space. However, within the region $c_1\eta\leq|X|\leq\eta$, we have
\[
	2\sigma_0\leq\theta(X)\leq \dfrac{A}{2}\eta^{\alpha}+2\sigma_0.
\]
Taking into account that the function $B = B_1$ set in \eqref{def B} is non-decreasing, we readily obtain
\[
\begin{array}{rll}
\beta(\theta(X)) & \geq & \gamma\theta(X)^{\gamma-1}B(2\sigma_0) \\
  & \geq & \gamma\theta(\eta)^{\gamma-1}B(2\sigma_0)  \\
 & \geq & \gamma\left(\dfrac{A}{2}\eta^{\alpha}+2\sigma_0 \right)^{\gamma-1}B(2\sigma_0).\\
\end{array}
\]
Therefore, taking $0<A\ll 1$ small enough,  
\[
	\gamma\left(\dfrac{A}{2}\alpha\eta^{\alpha}+2\sigma_0 \right)^{\gamma-1}B(2\sigma_0)>\dfrac{1}{2}\gamma(2\sigma_0)^{\gamma-1}B(2\sigma_0)>\Lambda NA\alpha^2\eta^{\alpha-2}.
\]
and we indeed obtain the desired pointwise inequality 
\[
F(D^2\theta)\leq \beta(\theta(X)),
\]
in the region $c_1\eta\leq|X|\leq\eta$. Let us turn our attention to the region $\eta\leq|X|$. Readily we have
\[
\theta_{ij}=A\alpha\left[(\alpha-2)|X|^{\alpha-4}X_iX_j+\delta_{ij}|X|^{\alpha-2}\right].
\]
Thus, at a point of the form $(|X|,0,\cdots,0)$, we obtain
\[
\begin{array}{rll}
\theta_{11}= & A\alpha(\alpha-1)|X|^{\alpha-2} &  \\
\theta_{ii}= & A\alpha|X|^{\alpha-2} & \mbox{if } i>1 \\  
\theta_{ij}= & 0 & \mbox{if  } i\neq j. \\
\end{array}
\]
Therefore, again by symmetric invariance of $\theta$ and ellipticity of $F$, we can write
\begin{equation}\label{B1}
F(D^2\theta(X))\leq\Lambda\left[A\alpha(\alpha-1)+(N-1)A\alpha\right]|X|^{\alpha-2}\leq\Lambda NA\alpha\eta^{\alpha-2}.
\end{equation}
On the other hand, in the region $\eta\leq |X|$, we have for $M\geq \sup\limits_{X\in \Omega}|X|$, that
\[
M^\alpha\geq |X|^\alpha-\dfrac{\eta^\alpha}{2}>0,
\]
and so,
\[
\beta(\theta(X)) \geq \gamma\left(A\left(|X|^\alpha-\dfrac{1}{2}\eta^ \alpha\right)+2\sigma_0\right)^{\gamma-1}B(\theta(\eta))> \gamma\left(AM^\alpha+2\sigma_0\right)^{\gamma-1}B(\theta(\eta)).
\]
Thus, adjusting $A>0$ even smaller, if necessary, we can assure 
$$
	A M^\alpha+2\sigma_0<4\sigma_0,
$$ 
and therefore,
\[
	\beta(\theta(X))> \gamma\left(4\sigma_0\right)^{\gamma-1}B(\theta(\eta)).
\]
Finally by $\ref{B1}$ and the inequality above, as well as diminishing the value of $A>0$ even further, if necessary, we reach
\[
	F(D^2\theta(X))\leq\Lambda NA\alpha|X|^{\alpha-2}<\gamma\left(4\sigma_0\right)^{\gamma-1}B(\theta(\eta))\leq\beta(\theta(X)).
\]
So its follow \textit{(3)}. By construction \textit{(2)} is valid, and the proof of Proposition \ref{prop barrier} follows.
\end{proof}

Proposition \ref{prop barrier} provides the existence of the appropriate barrier in the unit scale $\varepsilon = 1$. To furnish the desired supersolution for any $\varepsilon> 0$ small we argue as follows. Fixed $\varepsilon > 0$, we consider the fully nonlinear elliptic operator
$$
	F_\varepsilon(\mathcal{M}) := \varepsilon^{2-\alpha} F(\varepsilon^{\alpha-2}\mathcal{M}).
$$
It is standard to verify that $F_\varepsilon$ is uniform elliptic with the same ellipticity constants as $F$. Proposition \ref{prop barrier} applied to $F_\varepsilon$ provides a $C^{1,1}$ function $\theta = \theta(\varepsilon)$ that satisfies the differential inequality
\[
	F_\varepsilon(D^2\theta(X)) = \varepsilon^{2-\alpha} F(\varepsilon^{\alpha-2}D^2\theta(X))\leq \beta_1(\theta(X)).
\]
Finally, we define
\begin{equation}\label{thetae}
	\theta_{\varepsilon}(X):=\varepsilon^{\alpha}\theta(\varepsilon^{-1}X),
\end{equation}
where once more, $\alpha$ is the value set in \eqref{alpha}. We verify readily that $\theta_{\varepsilon}$ defined above satisfies

\begin{itemize}
\item[$\checkmark$] $\theta_\varepsilon = 2\sigma_0\varepsilon^\alpha$ in $B_{c_1\varepsilon\eta}$;
\item[$\checkmark$] $\theta_\varepsilon\geq c_2\eta^\alpha$ in $\Omega \setminus B_{\varepsilon\eta}$;
\item[$\checkmark$] $\theta_{\varepsilon}\in C^{1,1}(\Omega)$ and it is a supersolution to \eqref{Ee}.
\end{itemize}

We are ready to establish strong nondegeneracy of minimal solutions to the singularly perturbed problem \eqref{Ee}.

\begin{theorem}[Strong nondegeneracy] \label{lemma SN} Let $X_0 \in \{u_\varepsilon > \varepsilon^\alpha \}$. There exist two universal positive constants $c_0 >0$ and $r_0 > 0$ such that if $r< r_0$, there holds
 $$	
 	\sup\limits_{B_r(X_0)} u_\varepsilon \ge c_0 r^\alpha,
 $$	
 for $\alpha$ as in \eqref{alpha}.
 \end{theorem}
 
 \begin{proof}
Given $r < r_0$, we construct $\theta_\varepsilon$ for $\eta = r/\varepsilon$. By minimality of $u_\varepsilon$,
\[
u_\varepsilon(Z)>\theta_\varepsilon(Z),
\]
for some point $Z \in \partial B_r(X_0)$. Indeed, suppose for the sake of contradiction that $u_\varepsilon\leq \theta_\varepsilon$ along $\partial B_{r}$. Define
\[
w_\varepsilon=\left\{
\begin{array}{ccl}
\min\{\theta_\varepsilon,u_\varepsilon\} & \mbox{in} & \overline{B_{r}};\\
 u_\varepsilon & \mbox{in} & \Omega\setminus \overline{B_{r}}. \\
\end{array}
\right.
\]
Thus, $w_\varepsilon$ is supersolution to \eqref{Ee}; however in $B_{c_1r}$, we have,
\[
u_\varepsilon > \varepsilon^\alpha >2\sigma_0\varepsilon^\alpha \equiv \theta_\varepsilon = w_\varepsilon, 
\]
which contradicts the minimality of $u_\varepsilon$. In conclusion,
\[
c_2 r^\alpha \leq \theta_\varepsilon(Z)<u_\varepsilon(Z) \leq \sup_{B_r }u_\varepsilon,
\]
and the Theorem is proven.
 \end{proof}

An immediate Corollary of Theorem \ref{lemma SN} combined with Corollary \ref{gradest1} is the upper and lower control of $u_\varepsilon$ by $r^\alpha$ in $B_r \subset \{u_\varepsilon > \varepsilon^\alpha\}$. 

\begin{corollary}\label{strong col} Given a subdomain $\Omega'\Subset \Omega$, there exists a universal constant $C=C(\Omega')$ such that for $X_0\in \Omega'\cap\{u_\varepsilon > \varepsilon^\alpha \}$ and $r\leq r_0$,
\[
C^{-1} r^\alpha \leq \sup\limits_{B_r(X_0)} u_\varepsilon  \leq C(r^\alpha+u_\varepsilon(X_0)). 
\] 
 \end{corollary}

%%%%
%%%%

Recall we have set the following notation: $d_\varepsilon(X) = \text{dist}\left ( X, \partial \{u_\varepsilon > \varepsilon^\alpha \} \right )$. Our next step is to show that in fact $u_\varepsilon$ does growth at the sharp rate away from the free boundary, that is $\sim d_\varepsilon^\alpha$.
        
\begin{theorem}[Sharp Growth] \label{lemma Growth} Let $X_0 \in \{u_\varepsilon > \varepsilon^\alpha \}$. Then there exists $c_0 > 0$ universal such that 
 $$	
 	u_\varepsilon(X_0) \ge c_0 d_\varepsilon(X_0)^\alpha.
 $$	 
 \end{theorem}
 \begin{proof}
 Let us suppose for sake of contradiction that no such a constant exists. If so, there would exist a sequence of points $X_n \in  \{u_\varepsilon > \varepsilon^\alpha \}$, with $d_n := d_\varepsilon(X_n) \to 0$ and 
 $$
 	u_\varepsilon(X_n) \le \dfrac{1}{n} d_n^\alpha.
 $$
 Let us define 
 $$
 	v_n(Y) := \dfrac{1}{d_n^\alpha} u_\varepsilon(X_n + d_n Y).
 $$
 The function $v_n \ge 0 $ in $B_1$, and easily we verify that $v_n$ is a minimal solution to 
\begin{equation}\label{eq in n}
 	F_n(D^2 v_n) = \gamma v_n^{1-\gamma}B_{\frac{\varepsilon}{d_n}}(v_n) \quad \text{ in } B_1,
\end{equation}
where $F_n(\mathcal{M}):=d_n^{2-\alpha}F(d_n^{\alpha-2}\mathcal{M})$ for all $\mathcal{M} \in \mbox{Sym}(N)$ and $B_{\frac{\varepsilon}{d_n}}$ is the smooth approximation of $t^{+}$ set up in \eqref{def B}. From its very definition, we check that
\begin{equation}\label{properties Be}
B_{\frac{\varepsilon}{d_n}}(t)=
\left\{
\begin{array}{ccc}
	0 & \text{for} & 0\leq t \leq \sigma_0\left(\dfrac{\varepsilon}{d_n}\right)^{\alpha}, \\
	1 & \text{for} & t\leq(1-\sigma_0)\cdot\left(\dfrac{\varepsilon}{d_n}\right)^\alpha.\\
\end{array}
\right.
\end{equation}
  Since $F_n$ is uniformly elliptic with same ellipticity constants as $F$, we can apply Theorem \ref{lemma SN} to $v_n$ as to obtain 
 \begin{equation}\label{eq 01}
 	\sup\limits_{B_{\kappa}} v_n \ge c_0 \kappa^\alpha,
 \end{equation}
 for a universal constant $c_0$ and for any $\kappa > 0$. However, by Theorem $\eqref{regularitythm}$, there holds
$$
	v_n(X) \le C(\kappa^\alpha+v_n(0)), 
$$  
for a universal constant $C>0$ and for any $\kappa > 0$. In particular, for $\kappa_0 \ll 1$,
$$
	v_n(X) \le C\left(\frac{\sigma_0}{2C}\varepsilon^\alpha+v_n(0)\right) \quad \text{ in } B_{\kappa_0}.
$$
If we take $n \gg 1$, $v_n(0) \le \dfrac{\sigma_0}{2C}\varepsilon^\alpha$ and then
$$
	v_n(X) \le \sigma_0\varepsilon^\alpha\quad \text{ in } B_{\kappa_0}.
$$
In view of Equation \ref{eq in n} and \eqref{properties Be}, we see
$$
	F_n(D^2 v_n) = 0 \quad \text{ in } B_{\kappa_0},
$$
for $n \gg 1$. But then, by classical homogeneous Harnack inequality, see \cite{C2}, and strong nondegeneracy stated in \eqref{eq 01}
$$
	c_0 \left (\frac{\kappa_0}{2}  \right )^\alpha \le \sup\limits_{B_{\frac{\kappa_0}{2}}} v_n \le C v_n(0) = \text{o}(1),
$$
which  finally give us a contradiction.
 \end{proof}
 
An important consequence of Theorem \ref{lemma Growth} is the complete control of $u_\varepsilon(X)$ in terms of the $d_\varepsilon(X)^\alpha$. 

 \begin{corollary}\label{Growth col} Given a subdomain $\Omega'\Subset \Omega$, there exists a universal constant $C=C(\Omega')$ such that for $X\in \Omega'\cap\{u_\varepsilon > \varepsilon^\alpha \}$ and $\varepsilon \le d_\varepsilon(X)$,
\[
Cd_\varepsilon(X)^\alpha\geq u_\varepsilon(X) \ge C^{-1} d_\varepsilon(X)^\alpha.
\] 
 \end{corollary}
 
 \begin{proof}
 The inequality by below is precisely the statement of Theorem \ref{lemma Growth}. Now for  $Z\in\partial\{u_{\varepsilon}>\varepsilon^{\alpha}\}$, such that $|Z-X|=d_\varepsilon(X)$, it follows from Theorem $\ref{regularitythm}$
 \[
 u_\varepsilon(X)\leq \sup\limits_{B_{d_\varepsilon(X)}(Z)}u_\varepsilon \leq C\left(d_\varepsilon(X)^\alpha+\varepsilon^\alpha\right) \leq C d_\varepsilon(X)^\alpha,
 \]
 and the Corollary is proven.
 \end{proof}

 As usual a fine geometric control  as the one stated in Corollary \ref{Growth col} implies uniform positive density of the approximating region $\{u_\varepsilon > \varepsilon^\alpha \}$.
 
 \begin{corollary}\label{corollarydensity} Given a subdomain $\Omega'\Subset\Omega$, there exists constant $0<c\le 1$, depending only on $\Omega'$ and universal parameters, such that for any $X\in \Omega' \cap \{u_{\varepsilon}>\varepsilon ^{\alpha}\}$ and $\varepsilon\ll \delta$, we have
\[
	\dfrac{\Leb \left ( B_{\delta}(X)\cap\{u_{\varepsilon}> \varepsilon^\alpha\} \right )}{\Leb (B_\delta)} \geq c .
\]
\end{corollary}

\begin{proof}

By strong non-degeneracy there exists $Y_0\in \overline{B_{\delta}(X)}\cap\{u_\varepsilon> \varepsilon^\alpha\}$ such that
\[
	u_\varepsilon(Y_0)\geq c_0\delta^\alpha.
\]
Define $d(Y):=|Y_0-Y|$. By Corollary $\ref{strong col}$,
$$
u_\varepsilon(Y_0)\leq \sup\limits_{B_{d(Y)}(Y)}u_\varepsilon \leq c_1(d(Y)^\alpha+u_\varepsilon(Y)),
$$
and so,
$$
\dfrac{1}{c_1}\left( c_0\delta^\alpha - c_1d(Y)^\alpha \right)\leq u_\varepsilon(Y).
$$
Note that, we can choose $0\leq \tau \ll 1$ universally small such that 
$$Y\in B_{\tau\delta}(Y_0)\cap B_\delta(X) \quad \mbox{and} \quad
 u_\varepsilon(Y) \geq \varepsilon^\alpha.
$$
In conclusion, 
$$
	\Leb \left (B_{\delta}(X)\cap\{u_{\varepsilon}>\varepsilon^\alpha\} \right ) \geq \Leb \left ( B_{\delta}(X)\cap B_{\tau \delta}(Y_0) \right ) \geq c\delta^N.
$$
for a universal constant $c>0$. 
\end{proof}

\section{Harnack type inequalities} \label{Section Harnack}

It is well established that Harnack type inequalities are among the central properties of solutions to second order elliptic equations. For non-negative viscosity solutions to fully nonlinear equations with non-homogeneous right hand side, 
$$
	F(D^2v) = f(X), \quad Q_1
$$
Krylov-Safonov \cite{KS} and Caffarelli \cite{C2} (see also \cite{CC}, Chapter 4) proved the following sharp Harnack inequality:
\begin{equation}\label{classical hi}
	\sup\limits_{Q_{1/2}}  v \le C(n,\lambda, \Lambda) \left ( \inf\limits_{Q_{1/2}} v + \|f\|_{L^n(Q_1)} \right).
\end{equation} 
   As mentioned in previous Sections, one of the major mathematical difficulties in dealing with singular equations as in \eqref {Damiao_Eduardo Eq01} is the fact that right hand side blows-up near the quenching region. In particular, if one tries to interpret the singular term $\gamma u^{\gamma -1}$ as a right hand side $f(X)$ for the equation, classical Harnack inequality \eqref{classical hi} gives no information near the free boundary. 

\par

The key objective of this Section is to establish, uniform-in-$\varepsilon$ \textit{clean} geometric Harnack type inequalities for solutions to equation \eqref{Ee}. 

 \begin{theorem}[$L^1$-Harnack inequality] \label{int col} Given $\Omega'\Subset\Omega$, $X_0 \in \{u_\varepsilon > \varepsilon^\alpha \}\cap \Omega'$. Then
$$
	\intav{B_\rho(X_0)} u_\varepsilon\, dx \ge c \rho^\alpha,
$$
for a universal constant $c>0$, independent of $\varepsilon$. 
 \end{theorem}

 \begin{proof} From Lemma \ref{lemma SN},  there is a $Z \in \overline{B_\rho(X_0)}\cap\{u_\varepsilon>\varepsilon^\alpha\}$ and a $c_0>0$ universal, such that
 $$
 	u_ \varepsilon(Y_0) \ge c_0\rho^\alpha.
 $$
 As in the proof of the Corollary $\ref{corollarydensity}$, for $\theta \ll 1$ but universal, we obtain
 \[
 \begin{array}{lll}
		u_\varepsilon(Y) \geq C\rho^\alpha  & \mbox{in} & B_{\theta\rho}(Y_0). \\
\end{array}
 \] 
 Finally,
 $$
 	\intav{B_\rho(X_0)} u_\varepsilon\, dx \ge C_N \intav{B_\rho(X_0) \cap B_{\theta \rho}(Z)} u_\varepsilon\, dx  \ge C \rho^\alpha,
 $$
 for $C>0$ a universal constant. Thus, the proof is concluded.
 \end{proof}

Our next Theorem is a \textit{clean} Harnack inequality for ball touching the approximating free boundary  $\partial \{u_\varepsilon > \varepsilon^\alpha\}$. 

\begin{theorem}[Harnack Inequality for tangential balls]  Let $X_0 \in \{u_\varepsilon > \varepsilon^\alpha \}$ and $\varepsilon \le d :=d_\varepsilon(X_0)$. Then, there exist a universal constant $C>0$ such that
$$
	\sup\limits_{B_{\frac{d}{2}}(X_0)} u_\varepsilon \le C \inf\limits_{B_{\frac{d}{2}}(X_0)} u_\varepsilon.
$$
\end{theorem}
\begin{proof}
Let $\xi_0,\xi_1\in \overline{B_{\frac{d}{2}}(X_0)}$, such that
\[
\inf\limits_{B_{\frac{d}{2}}(X_0)}u_\varepsilon=u(\xi_0) \quad \mbox{and} \quad \sup\limits_{B_{\frac{d}{2}}(X_0)}u_\varepsilon=u(\xi_1).
\]
As $d_\varepsilon(\xi_0)\geq \dfrac{d}{2}$, by nondegeneracy 
\begin{equation}\label{N1}
u_\varepsilon(\xi_0)\geq C_1d^\alpha.
\end{equation}
By other hand, using the corollary $\ref{strong col}$, we get
\[
u_\varepsilon(\xi_1)\leq C_2\left(\dfrac{d^\alpha}{2}+u_\varepsilon(X_0)\right).
\]
Taking $Y\in\partial\left\{u_\varepsilon>\varepsilon^\alpha\right\}$, where $d=|X_0-Y|$, we have that 
\[
u_\varepsilon(X_0) \leq \sup\limits_{B_d(Y)} u_\varepsilon\leq C_2(d^\alpha+\varepsilon^\alpha)\leq C_3 d^\alpha.
\]
So, by the three last inequalities, we obtain
\[
	\sup\limits_{B_{\frac{d}{2}}(X_0)} u_\varepsilon \le C \inf\limits_{B_{\frac{d}{2}}(X_0)} u_\varepsilon.
\]
for a constant $C>0$ that does not depend of $u_\varepsilon$ and $\varepsilon$.
\end{proof}

\section{Hausdorff estimates of the free boundary} \label{Section Hausdorff}

In this section, we turn our attention to uniform geometric-measure properties of $\sim \varepsilon^\alpha$-level surfaces of $u_\varepsilon$. These surfaces approximate the limiting free boundary $\mathfrak{F} := \partial \{u>0 \} \cap \Omega$, where $u$ is the desired limiting function. Through this section we shall work under the following extra structural  condition on the operator $F$:
\begin{definition}\label{AL cond on F} We say a uniformly elliptic operator $F \colon \text{Sym}(N) \to \mathbb{R}$ is asymptotically concave if there exists a positive definite matrix $\mathcal{F}=\left(f_{ij}\right)_{ij}$ and a nonnegative constant $C_F \ge 0$ such that
	\begin{equation}\label{hip}\tag{AC}
		f_{ij}\mathcal{M} _{ij} - F(\mathcal{M} ) \ge - C_F,
	\end{equation}
for all matrix $\mathcal{M}  \in \text{Sym}(N)$,
\end{definition}
Initially, let us point out that indeed hypothesis \eqref{hip} is an asymptotic condition as $\|M\| \gg 1$, as it suffices to hold in the limit for $\|M\| \to +\infty$. It represents a sort of concavity condition at infinity of $F$. For concave operators, $C_F = 0$. The structural condition \eqref{hip} arises from recent considerations on the \textit{recession} operator
$$
	F^\star(\mathcal{M} ) := \lim\limits_{\mu \to 0} \mu F(\mu^{-1} \mathcal{M} ).
$$ 
The limiting operator $F^\star$ should be interpreted as the tangential equation for the natural elliptic scaling on $F$.
For example, for a number of elliptic operators, it is possible to verify the existence of the limit 
$$
	b_{ij} := \lim\limits_{\|\mathcal{M}\| \to \infty} F_{ij}(\mathcal{M}).
$$ 
In this case, $F^\star(\mathcal{M}) = \tr(b_{ij} \mathcal{M})$ and \eqref{hip} is automatically satisfied. A particularly interesting example is the class of Hessian operators of the form 
$$
	F_\iota(M) = f_\iota(\lambda_1, \lambda_2, \cdots \lambda_N) := \sum\limits_{j=1}^N (1 + \lambda_j^\iota)^{1/\iota},
$$
where $\iota$ is an odd natural number. For this family of operators, we have $F^\star_\iota = \Delta$
and condition \eqref{hip} is satisfied.

In \cite{RT} it is proven that the recession operator $F^\star$ rules the free boundary condition for fully nonlinear cavitation problems. In \cite{ST}, it is established  further regularity estimates of solutions to $F(X,D^2u) = f(X)$ via properties of the recession function. 
\par

Before continuing, let us make few remarks as to organize some systematic arguments that will appear within the next proofs. 

\begin{remark}\label{obs1}
Given $X_0\in \{u_\varepsilon > \varepsilon^\alpha\}$, where $u_\varepsilon(X_0)=C_1\varepsilon^{\alpha}$ for $C_1>1$, $\varepsilon\ll\rho$ and $\rho$ universally small, we have from Theorem \ref{regularitythm} that in $B_\rho(X_0)$, for $\rho \ll 1$ to be adjusted soon, there holds
\[
	u_\varepsilon^{\gamma-1}\geq C_2(\rho^\alpha+C_1\varepsilon^\alpha)^{\gamma-1}.
\] 
Therefore, if $\varepsilon$ is small as to 
$$
	\varepsilon^\alpha<\dfrac{\rho^\alpha}{C_1}
$$ 
and the radius $\rho$ is also selected universally small as to 
$$
	2C_2\rho^\alpha \leq \sqrt[\gamma-1]{\frac{2}{\gamma}C_F}
$$ 
we  readily obtain 
\[
	\gamma u_\varepsilon^{\gamma-1}\geq 2C_F \quad \text{in } B_\rho(X_0),
\]
for $C_F>0$ as in \eqref{hip}. Also, as 
$$
	(1-\sigma_0)\varepsilon^\alpha<C_1\varepsilon^\alpha,
$$ 
we have 
$$
	F(D^2u_\varepsilon)=\beta_\varepsilon(u_\varepsilon)=\gamma u_\varepsilon^{\gamma-1}, \quad \text{in } B_\rho(X_0).
$$ 
In conclusion, we obtain that $u_\varepsilon$ is a $f_{ij}$-subharmonic function in $B_\rho(X_0)$ for $\varepsilon \ll 1$, i.e.,
$$
	f_{ij}D_{ij}u_\varepsilon \geq F(D^2u_{\varepsilon})-C_F=\gamma u_\varepsilon^{\gamma-1}-C_F\geq 0,
$$
\end{remark}

We are now ready to establish the first Hausdorff type estimate for the level surface $\{u_\varepsilon \sim \varepsilon^\alpha\}$.

\begin{lemma}\label{hausdorff1}
Given a subdomain $\Omega'\Subset \Omega$, there exists a constant $C$ depending on $\Omega'$ and universal parameters such that, for $X_0\in\Omega'\cap \{u_\varepsilon> \varepsilon^\alpha\}$, with $u_\varepsilon(X_0)=C_1\varepsilon^{\alpha}$, with $C_1>1$ and $\varepsilon\ll \rho$ for $\rho$ universally small, there holds
\[
\int\limits_{\{C_1\varepsilon^\alpha<u_\varepsilon<\mu^\alpha\}\cap B_{\rho}(X_0)}u_{\varepsilon}^{-\gamma}|\nabla u_\varepsilon|^2 dX\leq C\mu\rho^{N-1},
\]
for a.e.  $0<\rho\ll 1$.
\end{lemma}

\begin{proof}
The proof starts off by verifying that for $\varepsilon$ and $\rho$ universally small, the following differential inequality holds:
\begin{equation}\label{A11}
\sum_{ij}f_{ij}D_{ij}(u_\varepsilon^{\frac{1}{\alpha}})\geq 0 \quad \mbox{ in }  \{u_\varepsilon>\varepsilon^\alpha\}\cap B_\rho(X_0).
\end{equation}
where $\mathcal{F}=\left(f_{ij}\right)_{ij}$, as in Definition $\eqref{AL cond on F}$.
To show such an estimate, we argue as follows: fix a non-singular linear operator $A \colon \mathbb{R}^N \to \mathbb{R}^N$. We compute
\begin{equation}\label{Aoperator}
\begin{array}{ccl}
\displaystyle \displaystyle\sum_{ij}f_{ij}D_{ij}(u_\varepsilon^{\frac{1}{\alpha}}) & = & \displaystyle \dfrac{1}{\alpha}u_\varepsilon^{\frac{1}{\alpha}-1}\tr\left(A^{-1}\mathcal{F}(A^{-1})^TD^2(u_\varepsilon\circ A)\right)\circ A^{-1} \\
 & + & \displaystyle \dfrac{1}{\alpha}\left(\dfrac{1}{\alpha}-1\right)u_\varepsilon^{\frac{1}{\alpha}-2}\tr(A^{-1}\mathcal{F}(A^{-1})^T\nabla(u_\varepsilon\circ A)\otimes \nabla(u_\varepsilon\circ A))\circ A^{-1}.
\end{array}
\end{equation}
In addition, $u_\varepsilon\circ A$ solves the following uniform elliptic fully nonlinear equation
\[
	F_A(D^2 v)=\gamma v^{\gamma-1},
\]
where the operator $F_A$ is given by 
$$
	F_A(\mathcal{M}) :=F((A^{-1})^T \mathcal{M} A^{-1}).
$$ 
Easily one verifies that $F_A$ is in fact uniformly elliptic, with the same ellipticity constants as $F$. Thus, by optimal gradient bounds, we obtain
\[
	|\nabla(u_\varepsilon\circ A)|^2\leq C(u_\varepsilon\circ A)^{\gamma}.
\]
and so, 
\begin{equation}\label{Areg}
	\dfrac{1}{\alpha}\left(\dfrac{1}{\alpha}-1\right)|\nabla(u_\varepsilon\circ A)|^2\geq \dfrac{C}{\alpha}\left(\dfrac{1}{\alpha}-1\right)(u_\varepsilon\circ A)^{\gamma}.
\end{equation}
From the structural assumption  \eqref{hip}, 
\begin{equation}\label{A}
\begin{array}{ccl}
\tr((A^{-1})^T\mathcal{F}A^{-1}D^2(u_\varepsilon\circ A)) & \geq & F((A^{-1})^TD^2(u_\varepsilon\circ A)A^{-1})-C_F \\
 & & \\
 & \geq & F((D^2u_\varepsilon)\circ A)-C_F \\
  & & \\
 & \geq & \gamma (u_\varepsilon\circ A)^{\gamma-1}-C_F .
\end{array}
\end{equation}
Thus, if we select  $A$ as to satisfies 
$$
	\mathcal{F}=\frac{1}{C}\,AA^T
$$ 
where $C>0$ is the constant of inequality \eqref{Areg}, and  combine $\ref{Areg}$, $\ref{A}$ and $\ref{Aoperator}$, we end up with
\[
\sum_{ij}f_{ij}D_{ij}(u_\varepsilon^{\frac{1}{\alpha}})\geq \dfrac{1}{\alpha}u_\varepsilon^{\frac{1}{\alpha}-1}\left(\dfrac{\gamma}{2}u_\varepsilon^{\gamma-1}-C_F\right).
\]
Finally from Remark \ref{obs1}, we deduce that for $\varepsilon$ and $\rho$ universally small, the differential inequality \eqref{A11} indeed holds true. 
\par
We now continue with the proof of Lemma \ref{hausdorff1}. Define the following cut off function,
\[
\Phi=\left\{
\begin{array}{ccl}
\varepsilon\sqrt[\alpha]{C_1} & \mbox{in} & \{u_\varepsilon\leq C_1\varepsilon^\alpha\}; \\
u_{\varepsilon}^{\frac{1}{\alpha}} & \mbox{in} & \{C_1\varepsilon^\alpha< u_\varepsilon\leq\mu^\alpha\}; \\
\mu & \mbox{in} & \{u_\varepsilon>\mu^\alpha\}.
\end{array}
\right.
\]
Clearly we have
\begin{equation}\label{edu-eq-ins001}
	\displaystyle \int\limits_{\{C_1\varepsilon^\alpha< u_\varepsilon\leq\mu^\alpha\}\cap B_\rho(X_0)} f_{ij}	(u_\varepsilon^{\frac{1}{\alpha}})_i \cdot (u_\varepsilon^{\frac{1}{\alpha}})_j ~ dX = \displaystyle \int\limits_{B_\rho(X_0)}f_{ij}\Phi_i(u_\varepsilon^{\frac{1}{\alpha}})_j ~dX.
\end{equation}
Standard integrations by parts yield
\begin{equation}\label{edu-eq-ins002}
\begin{array}{lll}
\displaystyle \int\limits_{B_\rho(X_0)}f_{ij}\Phi_i(u_\varepsilon^{\frac{1}{\alpha}})_j ~ dX &=& \displaystyle -\int\limits_{B_\rho(X_0)}\Phi \cdot f_{ij} (u_\varepsilon^{\frac{1}{\alpha}})_{ij} ~ dX  \\
&+& \displaystyle \frac{1}{\rho}\int\limits_{\partial B_\rho(X_0)}f_{ij}\phi(u_\varepsilon^{\frac{1}{\alpha}})_i\cdot(x^j-x_0^j)\,d\mathcal{H}^{N-1}.
\end{array}
\end{equation}
Therefore, from the differential inequality established in \eqref{A11}, we conclude
\[
\int\limits_{\{C_1\varepsilon^\alpha< u_\varepsilon\leq\mu^\alpha\}\cap B_\rho(X_0)}f_{ij}(u_\varepsilon^{\frac{1}{\alpha}})_i(u_\varepsilon^{\frac{1}{\alpha}})_j dX
 \leq \frac{1}{\rho}\int\limits_{\partial B_\rho(X_0)} \Phi \cdot f_{ij} (u_\varepsilon^{\frac{1}{\alpha}})_i\cdot(x^j-x_0^j)\,d\mathcal{H}^{N-1}.
\]
Passing the derivatives through, we can further write the above estimate as
\[
\int\limits_{\{C_1\varepsilon^\alpha< u_\varepsilon\leq\mu^\alpha\}\cap B_\rho(X_0)}f_{ij} u_\varepsilon^{-\gamma}D_iu_\varepsilon D_ju_\varepsilon\, dx
\leq \frac{\alpha}{\rho}\int\limits_{\partial B_\rho(X_0)}\Phi\cdot f_{ij}u_\varepsilon^{-\frac{\gamma}{2}}D_iu_\varepsilon\cdot(x^j-x_0^j)d\mathcal{H}^{N-1}.
\]
Finally, from uniform ellipticity and optimal regularity of $u_\varepsilon$, we derive
\[
\int\limits_{\{C_1\varepsilon^\alpha< u_\varepsilon\leq\mu^\alpha\}\cap B_\rho(X_0)}u^{-\gamma}|\nabla u|^2 dX
\leq C\mu\rho^{N-1},
\]
as desired.
\end{proof}

For the next result, let is recall the following classical notation: given a set  $G \subset \mathbb{R}^N$, we will denote
\[
\mathcal{N}_\delta(G):=\{X\in \mathbb{R}^N \mid \dist(X,G)<\delta\}.
\]

In the sequel we show the main step towards uniform bounds of the $\mathcal{H}^{N-1}$-Hausdorff measure of the level-surfaces $\{u_\varepsilon > \varepsilon^\alpha\}$.

\begin{lemma}\label{principallemma}
Fixed $\Omega'\Subset\Omega$, there exists a constant $C^\star$ that depends only on $\Omega'$ and universal parameters such that if,
\[
C^\star\mu \leq 2\rho \leq \dfrac{\dist(\Omega',\partial\Omega)}{10}
\]
then, for $\mu,\varepsilon>0$ universally small and $\mu \ll \rho$, for $\rho$ also universally small, we have
\[
 \Leb \left ( \{C_1\varepsilon^\alpha<u_\varepsilon<\mu^\alpha \}\cap B_\rho(X_0) \right ) \leq \bar{C}\mu\rho^{N-1},
\]
where again $\bar{C}=\bar{C}(\Omega')$ depends only on $\Omega'$ and universal constants and $X_0\in\Omega'\cap\partial\{C_1\varepsilon^\alpha<u_\varepsilon<\mu^\alpha\}$, with 
$d_\varepsilon(X_0)\leq \frac{\dist(\Omega',\partial\Omega)}{10}$ and $C_1>1$.
\end{lemma}
\begin{proof}
Let $\{B_j\}$ be a finite family of balls covering $\partial\{C_1\varepsilon^\alpha<u_\varepsilon\}\cap B_{\rho}(X_0)$, with radius constant equal to $C^\star\mu$ and center $X_j\in \partial\{C_1\varepsilon^\alpha<u_\varepsilon\}\cap B_{\rho}(X_0)$, where $C^\star$ will be chosen \textit{a posteriori}. By Heine-Borel Lemma, there exists a universal constant $m$ such that
\[
\sum\limits_j \chi_{B_j}\leq m.
\]
We can assure that
\[
\bigcup_j B_j \subset \left[\mathcal{N}_{\frac{d}{8}}(\Omega')\cap B_{4\rho}(X_0)\right].
\]
where $d:=\dist(\Omega',\partial\Omega)$. As in the proof of Lemma \ref{hausdorff1},  we consider
\[
\Phi=\left\{
\begin{array}{ccl}
\varepsilon\sqrt[\alpha]{C_1} & \mbox{in} & \{u_\varepsilon\leq C_1\varepsilon^\alpha\}; \\
u_{\varepsilon}^{\frac{1}{\alpha}} & \mbox{in} & \{C_1\varepsilon^\alpha< u_\varepsilon\leq\mu^\alpha\}; \\
\mu & \mbox{in} & \{u_\varepsilon>\mu^\alpha\}.
\end{array}
\right.
\]
We now claim that is possible to find, for each $j$, balls $B_j^1$ and  $B_j^2$ both contained in $B_j$, satisfying:
\begin{description}
\item[(1)] the radius of $B_j^1$ and  $B_j^2$ are in order $\mu$ (up to  universal contraction)
\item[(2)] $\Phi\geq \sqrt[\alpha]{\dfrac{3}{4}}\mu$ in $B_j^1$ and $\Phi\leq \sqrt[\alpha]{\dfrac{2}{3}}\mu$ in $B_j^2$.
\end{description}

To show the above claim, we argue as follows:  take $X_1\in \frac{1}{4}\overline{B_j}$, such that 
$$
	u_\varepsilon(X_1)=\sup\limits_{\frac{1}{4}B_j}u_\varepsilon.
$$ 
By strong nondegeneracy, 
\[
u_\varepsilon(X_1)\geq c_0\left(\frac{C^\star\mu}{4}\right)^\alpha \ge C_2 \mu^\alpha,
\]
if $C^\star \gg 1$ is chosen universally large enough, where $C_2>0$ is a constant obteined from Theorem \ref{regularitythm}. This last theorem, given $X\in B_j$, we obtain
\begin{equation}\label{h2}
u_\varepsilon (X)  \geq  \frac{1}{C_2}u(X_1)-|X-X_1|^\alpha
  \geq  \mu^\alpha - |X-X_1|^\alpha.
\end{equation}
Taking 
\[
|X-X_1|<\sqrt[\alpha]{\frac{1}{4}}\mu,
\]
we obtain
\[
\Phi^\alpha(X)=u_\varepsilon(X)\geq \frac{3}{4}\mu^\alpha \quad \mbox{in } B^1_j:=B_{r^1_j}(X_1)
\]
where $r^1_j:=\sqrt[\alpha]{\frac{1}{4}}$ is a universal constant. To finish up the proof of this first statement, we just choose $C^\star$ large enough as to 
\[
\tilde{C_1}\ll C^\star \quad \Longrightarrow \quad B_j^1\subset B_j.
\]
Notice again that such a selection is universal. Similarly, for $B_j^2:=B_{r^2_j}(X_j)$ where $r^2_j:=\tilde{C_2}\mu\ll C^\star\mu$, we have
\[
\Phi(X)^\alpha=u_\varepsilon \leq \frac{2}{3}\mu^\alpha.
\]
From property $(2)$ proven above, assures the existence of a universal constant $\kappa > 0$ such that, for each $j \in \mathbb{N}$ 
\[
|\Phi-m_j|>\kappa\mu,
\]
in at least one the two balls $B_j^1,B_j^2\subset B_j$, where 
$$
	m_j:= \intav{B_j} \Phi(X) dX.
$$
Thus, by classical Poincar\'e inequality in balls, we derive
\[
\kappa^2\mu^ 2\leq \displaystyle\frac{1}{|B_j|}\int_{B_j}|\Phi-m_j|^2 dX \leq C_3\mu^ 2\displaystyle\frac{1}{|B_j|}\int_{B_j}|\nabla\Phi|^2 dX,
\]
which in turn gives
\[
\int_{\{C_1\varepsilon^\alpha<u_\varepsilon<\mu^\alpha \}\cap B_j}u_\varepsilon^{-{\gamma}}|\nabla u_\varepsilon|^2 dx\geq C_4|B_j|.
\]
In addition, by  nondegeneracy,  for all  $Y\in \{C_1\varepsilon^\alpha<u_\varepsilon<\mu^\alpha \}\cap B_\rho(X_0)$, we have
\[
C_5d_\varepsilon(Y)^\alpha \leq u_\varepsilon(Y) \leq \mu^\alpha.
\]
Hence
\[
\{C_1\varepsilon^\alpha<u_\varepsilon<\mu^\alpha \} \cap B_\rho(X_0) \subset \mathcal{N}_{\frac{1}{C_6}\mu}\left(\partial\{C_1\varepsilon^\alpha<u_\varepsilon\}\cap B_{2\rho}(X_0)\right),
\]
for $C_6=\sqrt[\alpha]{C_5}$. Thus, for $\mu\ll \rho$, and $C^\star \gg 1$, both universal, we reach
\[
\{C_1\varepsilon^\alpha<u_\varepsilon<\mu^\alpha \} \cap B_\rho(X_0) \subset 
 \bigcup 2B_j \subset B_{4\rho}(X_0).
\]
Finally,  applying Lemma \eqref{hausdorff1}  and taking into account the inclusion above, we estimate
\[
\begin{array}{ccl}
C_7\,\mu\rho^{N-1} & \geq & \displaystyle \int_{B_{4\rho}(X_0)\cap\{C_1\varepsilon^\alpha<u_\varepsilon<\mu^\alpha \}}u_\varepsilon^{-{\gamma}}|\nabla u_\varepsilon|^2 dx\\
 & & \\
  & \geq & \dfrac{1}{m}\displaystyle\sum\int_{2B_j\cap\{C_1\varepsilon^\alpha<u_\varepsilon<\mu^\alpha \}}u_\varepsilon^{-{\gamma}}|\nabla u_\varepsilon|^2 dx \\
  & & \\
  & \geq & \displaystyle\frac{C_4}{m}\sum|B_j| \\
  & & \\
  & \geq & \displaystyle\frac{C_4}{m}|B_\rho(X_0)\cup \{C_1\varepsilon^\alpha<u_\varepsilon<\mu^ \alpha\}|,
\end{array}
\]
where $C_7$ and $C_4$ are universal constants, which completes the proof of Lemma.
\end{proof}

In the sequel, we recall the definition of $\delta$-density.
\begin{definition}
Given an open subset $G$ of $\mathbb{R}^N$, we say that $G$ has the $\delta-$\textit{density property} in $\Omega$ for $0<\delta<1$, if there exists $\tau>0$ such that
\[
\dfrac{ \Leb \left ( B_\delta(X)\cap G \right ) }{ \Leb \left ( B_\delta(X) \right )}\geq \tau
\]
for all $X\in \partial G\cap \Omega$. If the property above is valid for any $0<\delta<1$, we say that $G$ has \textit{uniform density in}  $\Omega$ \textit{along} $\partial G$.
\end{definition}

Here we state a, by now, classical result from measure theory.
\begin{lemma}\label{lemmacited}
Given an open set $A\Subset \Omega$, there holds:
\begin{description}
\item [a)] If there exists $\delta$ such that $A$ has the $\delta-$density property, then there exists a constant $C=C(\tau,N)$, where:
\[
|\mathcal{N}_\delta(\partial A)\cap B_\rho(X)|\leq \frac{1}{2^N\tau}|\mathcal{N}_\delta(\partial A)\cap B_\rho(X)\cap A|+C\delta\rho^{N-1}
\]
with $X\in\partial A\cap\Omega$ and $\delta \ll \rho$.
\item [b)] If $A$ has uniform density in $\Omega$ along $A$, then $|\partial A\cap \Omega|=0$.
\end{description}
\end{lemma}

We are ready to state and prove the main result on this section. 

\begin{theorem}\label{haus1}
Given $\Omega'\Subset\Omega$ there exists a universal constant $C=C(\Omega')>0$, such that
\[
 \Leb \left ( \mathcal{N}_{\mu}(\{C_1\varepsilon^ \alpha< u_\varepsilon\})\cap B_\rho(X_0) \right ) \leq C\mu\rho^{N-1},
\]
whenever, $C_1>1$, $X_0\in\Omega'\cap\partial\{C_1\varepsilon^ \alpha<u_\varepsilon\}$, $d_\varepsilon(X_0)<\frac{1}{10}\dist(\Omega',\partial\Omega)$, $\mu\ll \rho$ with $\rho$ universally small and $C_1\varepsilon^\alpha<\mu^\alpha$. In particular,
\[
\mathcal{H}^{N-1}(\partial\{C_1\varepsilon^ \alpha<u_\varepsilon\}\cap B_\rho(X_0))\leq C\rho^{N-1}.
\]

\end{theorem}

\begin{proof}
Take $\mu=\delta$, is as in the statement  of Corollary \eqref{corollarydensity}. We have,
\[
\frac{ \Leb \left ( B_\delta(X)\cap\{u_\varepsilon>C_1\varepsilon^\alpha\} \right ) }{\Leb \left ( B_\delta(X) \right )}\geq C_2,
\] 
for $X\in \partial \{u_\varepsilon>C_1\varepsilon^\alpha\}$. We conclude that, $\partial\{u_\varepsilon>C_1\varepsilon^\alpha\}$ has the $\delta-$density property, and by Lemma \ref{lemmacited}, for a universal constant $M>0$, there holds
\begin{equation}\label{set1}
\begin{array}{lll}
 \Leb \left ( \mathcal{N}_\delta(\partial\{u_\varepsilon>C_1\varepsilon^\alpha\})\cap B_\rho(X_0) \right )
&\leq& \dfrac{1}{2^NC_2} \Leb \left ( \mathcal{N}_\delta(\partial\{u_\varepsilon>C_1\varepsilon^\alpha\})\cap B_\rho(X_0)\cap \{u_\varepsilon>C_1\varepsilon^\alpha\} \right ) \\
  &+& M\delta\rho^{N-1}.
\end{array}
\end{equation}
From Theorem $\ref{regularitythm}$, given $Y\in \mathcal{N}_\delta(\partial\{u_\varepsilon>C_1\varepsilon^\alpha\})\cap B_\rho(X_0)\cap \{u_\varepsilon>C_1\varepsilon^\alpha\}$ and $Z\in \partial\{u_\varepsilon>C_1\varepsilon^\alpha\}$, we can estimate
\[
\begin{array}{ccl}
u(Y) & \leq & C_3(|Z-Y|^\alpha+u(Z)) \\
     & \leq & C_3(\delta^\alpha+\mu^\alpha)\\
     & \leq & D\mu^ \alpha,
\end{array}
\]
where the last inequality, follows from $C_1\varepsilon^\alpha<\mu$ and $\delta=C\mu$. We have verified there exists $D>0$ universal, such that
\begin{equation}\label{set2}
\mathcal{N}_\delta(\partial\{u_\varepsilon>C_1\varepsilon^\alpha\})\cap B_\rho(X_0)\cap \{u_\varepsilon>C_1\varepsilon^\alpha\}\,\subset\, \{C_1\varepsilon^\alpha<u_\varepsilon<D\mu^\alpha\}\cap B_\rho(X_0).
\end{equation}

Finally, from Lemma \eqref{principallemma}, we conclude 
\begin{equation}\label{set03}
\Leb \left ( \{C_1\varepsilon^\alpha<u_\varepsilon<D\mu^\alpha\}\cap B_\rho(X_0) \right ) \leq C_4\mu\rho^{N-1},
\end{equation}
thus, combing \eqref{set1}, \eqref{set2}  and \eqref{set03} we reach,
\[
\begin{array}{ccl}
\Leb \left ( \mathcal{N}_\mu(\{u_\varepsilon>C_1\varepsilon^\alpha\})\cap B_\rho(X_0) \right ) & \leq & C_4\,\mu\rho^{N-1}.
\end{array}
\]

To conclude the proof of the $\mathcal{H}^{N-1}$ Hausdorff measure estimate, let $\{B_j\}$ be a covering of $\partial\{C_1\varepsilon^ \alpha< u_\varepsilon\}\cap B_\rho(X_0)$, where each ball be centered in $\partial\{C_1\varepsilon^ \alpha< u_\varepsilon\}\cap B_\rho(X_0)$  with radius $\mu$. We can write
\[
\bigcup B_j\subset \mathcal{N}_{\mu}(\{C_1\varepsilon^ \alpha< u_\varepsilon\})\cap B_{\rho+\mu}(X_0).
\]

Thus there exist dimensional constants $C_5,C_6>0$, such that
\[
\begin{array}{ccl}
\mathcal{H}_\mu^{N-1}(\partial\{C_1\varepsilon^ \alpha<u_\varepsilon\}\cap B_\rho(X_0)) & \leq & C_5\sum \mbox{Area}(\partial B_j) \\
&=&\dfrac{C_5}{\mu}\sum \Leb \left ( B_j \right ) \\
 & & \\
 & \leq & \dfrac{C_6}{\mu} \Leb \left ( \mathcal{N}_{\mu}(\{C_1\varepsilon^ \alpha< u_\varepsilon\})\cap B_{\rho+\mu}(X_0)\right ) \\
 & & \\
 & \leq &  C_6C_4(\rho+\mu)^{N-1}=C_6C_4\,\rho^{N-1}+\text{o}(1). \\
\end{array}
\] 
Letting $\mu\to 0$, we finish the proof of the Theorem.
\end{proof}

\section{Limiting free boundary problem}

In this Section, we address the fully nonlinear free boundary problem obtained by letting $\varepsilon \to 0$. The ultimate goal is to find a solution to the free boundary problem \eqref{Damiao_Eduardo Eq02} that enjoys all the desired analytic and geometric properties. 

Our analysis starts off by the compactness of minimal solutions to Equation \eqref{Ee}. In fact, Proposition \eqref{equicontinuity} implies that $\left \{u_\varepsilon \right \}_{\varepsilon > 0}$ is a compact sequence and up to a subsequence,
\begin{equation}\label{def u0}
	\lim\limits_{\varepsilon \to 0} u_\varepsilon =: u_0.
\end{equation}
This Section is devoted to the study of the limiting function $u_0$ and the free boundary problem it solves.
 
For the readers convenience, let us hereafter set up the following notations that we will use throughout this Section:
\[
\begin{array}{rll}
\{u_0>0\} & := & \{x\in \Omega \mid u_0(x)>0\}, \\
\mathfrak{F}(u_0) & := & \partial\{u_0>0\} \cap \Omega, \\
d_0(X) & := & \dist(X,\mathfrak{F}(u_0)).
\end{array}
\]

Next Theorem recovers the fully nonlinear equation satisfies by $u_0$ within its positive set as well as its precise growth behavior near the free boundary, $\mathfrak{F}(u_0)$. 

\begin{theorem} \label{u0 initial prop}
The limiting function $u_0$ defined in \eqref{def u0} is a viscosity solution to
\begin{equation}\label{solutionzero}
	F(D^2u)=\gamma u^{\gamma-1} \quad \mbox{in }  \{u>0\}.
\end{equation}
Moreover, for a fixed  $\Omega'\Subset\Omega$,  there exists a constant $C=C(\Omega')$ that depends on $\Omega'$ and universal constants such that for any $X\in \Omega'\cap \{u_0>0\}$, there holds
\[
	Cd_0(X)^\alpha\leq u_0(X) \leq C^{-1}d_0(X)^\alpha,
\]
whenever $d_0(X)\leq\frac{\dist(\Omega',\partial\Omega)}{4}$. In particular, $u_0$ is a solution to the free boundary problem \eqref{Damiao_Eduardo Eq02}.
\end{theorem}

\begin{proof}
Let us fix a point $X_0\in \{u_0>0\}$ and let $u_0(X_0):=\sigma>0$. By continuity 
$u_0\geq \frac{1}{2}\sigma$ in $B_\rho(X_0)$ for same $\rho>0$.  Since $u_\varepsilon\to u_0$ uniformly over compact sets, for $\varepsilon\ll 1$ we have
\[
u_\varepsilon \geq \dfrac{1}{8}\sigma>(1+\sigma_0)\varepsilon^\alpha.
\]
That is,  $u_\varepsilon$ satisfies 
\[
	F(D^2u_\varepsilon)=\gamma u_\varepsilon^{\gamma-1} \quad \mbox{in} \quad B_{\frac{1}{2}\rho}(X_0).
\]
By the stability of viscosity solutions under uniform limits, we conclude $u_0$ is indeed a viscosity solution to Equation \eqref{solutionzero}.

Let us now turn our attention to the growth rate controls. For that, fix $X_0\in \Omega'\cap\{u_0>0\}$, with $d_0(X_0)\leq \frac{1}{4}\dist(\Omega',\partial\Omega)$ and label $u_0(X_0)=s>0$.  For $\varepsilon\ll 1$  we have
\[
	u_\varepsilon(X_0)\geq \frac{s}{2}>\varepsilon^\alpha.
\]
Thus, according to Corollary $\ref{lemma Growth}$, we obtain
\[
u_\varepsilon(X_0)\geq Cd_\varepsilon(X_0)^\alpha.
\]
Let $Y_\varepsilon\in\partial\{u_\varepsilon>\varepsilon^\alpha\}$ be such that $d_\varepsilon(X_0)=|X_0-Y_\varepsilon|$. By uniform convergence, it clearly follows that $Y_\varepsilon\to Y_0$ and $u_0(Y_0)=0$. In conclusion,
\[
	u_0(X_0)\geq C|X_0-Y_0|^\alpha\geq Cd_0(X_0)^\alpha.
\]
The upper estimate is obtained similarly.
\end{proof}

Strong nondegeneracy property established for the approximating solutions $u_\varepsilon$ also passes to the  limiting configuration. 
\begin{theorem}\label{strongzero}
Given $\Omega'\Subset\Omega$, there exist universal constants $C,\rho_0>0$, depending only on $\Omega'$ and universal constants, such that for any $X\in\Omega'\cap\overline{\{u_0>0\}}$, $\rho\leq\rho_0$ and $d_0(X)<\frac{\dist(X,\partial\Omega')}{2}$, there holds
\[
	C^{-1}\rho^\alpha \leq \sup_{B_{\rho}(X)} u_0 \leq C(\rho^\alpha+u_0(X))
\]

\end{theorem}
%%%%

The proof of Theorem \ref{strongzero} is very similar to the one presented for Theorem \ref{u0 initial prop} and therefore,we shall omit the details. Next we show  the approximating configurations $\{u_\varepsilon\}$ converge to the liming one, $\{u_0 > 0 \}$ in the Hausdorff metric.

\begin{theorem}\label{uzero1}
Given $\delta>0$ and $\varepsilon\ll1$, we following inclusions hold:
\[
\begin{array}{c}
\{u_0>0\}\cap\Omega' \subset \mathcal{N}_\delta \left ( \{u_\varepsilon>C_1\varepsilon^\alpha\} \right ) \cap\Omega' \quad \text{and} \quad  \{u_\varepsilon>C_1\varepsilon^\alpha\} \cap\Omega' \subset  \mathcal{N}_\delta \left ( \{u_0>0\} \right )\cap\Omega'.
\end{array} 
\]
\end{theorem}

\begin{proof}
We will show only the last inclusion, as the first follows similarly.  Suppose, for the purpose of contradiction, that such inclusion is false. There would  exist, therefore, $\delta_0 > 0$ and a sequence of points $\{X_\varepsilon\}$, satisfying
\begin{enumerate}
\item[a)] $X_\varepsilon\in \Omega'\cap \{u_\varepsilon>C_1\varepsilon^\alpha\}$;
\item[b)] $\dist(X_\varepsilon,\{u_0>0\})>\delta_0$;
\item[c)]  $X_\varepsilon\to X_0$, and $\dist(X_0,\{u_0>0\})>\delta_0$.
\end{enumerate}
From property c) $u_0(X_0) = 0$. However, by strong non-degeneracy, Theorem \eqref{strongzero},  for each $\varepsilon$, we can find $Z_\varepsilon\in\overline{B_{\frac{1}{2}\delta_0}(X_\varepsilon)}$, such that
\begin{equation}\label{proof HM eq01}
	u_\varepsilon(Z_\varepsilon)=\sup\limits_{B_{\frac{1}{2}\delta_0}(X_\varepsilon)}u_\varepsilon \geq C\,\delta_0^\alpha.
\end{equation}
As $\varepsilon \to 0$, up to a subsequence, $Z_\varepsilon \to Z_0$. However, from \eqref{proof HM eq01} and $u_0 (Z_0) > 0$ and by property c) above $Z_0 \in \{u_0 = 0\}$, which is a contradiction.
\end{proof}

%%%%
It also follows as in Corollary \ref{corollarydensity} that the  set $\{u_0>0\}$ has uniform positive density along the free boundary $\mathfrak{F}(u_0)$. 
\begin{theorem}\label{uniform den u0}
 Given $\Omega'\Subset \Omega$ there exists a constant $0< c\le 1$, depending on $\Omega'$ and universal parameters, such that 
\[
\dfrac{ \Leb \left ( B_\delta(X)\cap \{u_0>0\} \right ) }{\Leb \left ( B_\delta(X) \right ) }\geq c,
\]
for all $X\in\mathfrak{F}(u_0) \cap \Omega'$.
\end{theorem}

The proof of Theorem \ref{uniform den u0} is  similar to the one presented for Corollary \ref{corollarydensity} and therefore we omit the details.

%%%%
 
 As a consequence of the analysis carried out in Section \ref{Section Harnack}, we will show that a \textit{clean} Harnack inequality is valid near the free boundary $\mathfrak{F}(u_0)$. As mentioned in that Section, such a result is quite surprising a first view, as the nonlinear source of the equation is of order $\sim u^{\gamma -1}$ and thus it blows up near the boundary of the quenching region. 

\begin{theorem}[Harnack Inequality for tangential balls]  Let $X_0 \in \{u_0 > 0 \}$ and $d:=d_0(X_0)$. Then, there exist a universal constant $C>0$ such that
$$
	\sup\limits_{B_{\frac{d}{2}}(X_0)} u_0 \le C \inf\limits_{B_{\frac{d}{2}}(X_0)} u_0.
$$
\end{theorem}
\begin{proof}
Let $\xi_0,\xi_1\in \overline{B_{\frac{d}{2}}(X_0)}$, be such that
\[
\inf\limits_{B_{\frac{d}{2}}(X_0)}u_0=u_0(\xi_0) \quad \mbox{and} \quad \sup\limits_{B_{\frac{d}{2}}(X_0)}u_0=u_0(\xi_1).
\]
Since $d_0(\xi_0)\geq \dfrac{d}{2}$, by Theorem $\ref{strongzero}$, there holds
\begin{equation}
u_0(\xi_0)\geq C_1d^\alpha.
\end{equation}
On the other hand, from Theorem \ref{strongzero}, we have
\[
u_0(\xi_1)\leq C_2\left(\frac{d^\alpha}{2}+u_0(X_0)\right),
\]
and for $Y\in\partial\left\{u_0>0\right\}$ where $d=|Y-X_0|$,   
\[
u_0(X_0) \leq \sup\limits_{B_d(Y)}u_0\leq  C_2d^\alpha.
\]
Thus,
\[
	\sup\limits_{B_{\frac{d}{2}}(X_0)} u_0 \le C_2 \inf\limits_{B_{\frac{d}{2}}(X_0)} u_0.
\]
for a constant $C_2>0$ that does not depend of $u_0$.
\end{proof}

As in the proof of Corollary $\ref{int col}$, we can establish a lower bound for solid integrals for $u_0$: for all $X_0\in \partial\{u_0>0\}\cap \Omega'$
\begin{equation}\label{inteq2}
	C_1\rho^\alpha \leq \intav{B_\rho(X_0)} u_0 dx,
\end{equation}
where $C_1=C_1(\Omega')>0$. Next we establish upper and lower control on spherical integrals of $u_0$. 

\begin{theorem}
Given $\Omega'\Subset \Omega$, there exists a universal constant $C=C(\Omega')$, such that for all $X_0\in \partial\{u_0>0\}\cap \Omega'$,
\[
C^{-1}\rho^\alpha \leq \intav{\partial B_\rho(X_0)} u_0 d\mathcal{H}^{N-1} \le C \rho^\alpha.
\]
\end{theorem}
\begin{proof}
The upper estimate follows directly from Corollary \eqref{strongzero}. We will show the lower bound by means of contradiction. Suppose the lower inequality is not valid. There would then exist $\rho_m>0$ and $X_m\in\partial\{u_0>0\}$, such that
\begin{equation}\label{inteq00}
	\dfrac{1}{\rho_m^\alpha}\intav{\partial B_{\rho_m}(X_m)} u_0 d\mathcal{H}^{N-1}=\text{o}(1),
\end{equation}
as $m\to \infty$. Clearly, \eqref{inteq00} implies 
\begin{equation}\label{inteq1}
\dfrac{1}{\rho_m^\alpha}\intav{\partial B_{r\rho_m}(X_m)} u_0 d\mathcal{H}^{N-1}=\text{o}(1),
\end{equation}
for all $0<r\leq1$. Define
\[
v_m(X):=\rho_m^{-\alpha}u_0(X_m+\rho_m X).
\]
Up to a subsequence, $v_m$ converges uniformly over compact subsets of $\mathbb{R}^N$, to a  function $v_0$. Furthermore, 
\begin{equation}\label{inteq3}
F_m(D^2 v_m)=\gamma v_m^{\gamma-1} \quad \mbox{in }  \{v_m > 0\},
\end{equation}
where $F_m (\mathcal{M}) := \rho_m^{\alpha} F( \rho_m^{-\alpha}  \mathcal{M})$. For any $0 < r \leq 1$,
\[
\intav{\partial B_{r\rho_m}(X_m)}u_0(Y)d\mathcal{H}^{N-1}=\intav{\partial B_{r}(0)}u_0(X_m+\rho_m X)d\mathcal{H}^{N-1}=\rho_m^\alpha\intav{\partial B_{r}(0)}v_m(X)d\mathcal{H}^{N-1}.
\]
Thus, by \eqref{inteq1}, letting $m\to\infty$, yields 
\[
	\intav{\partial B_{r}(0)}v_0 \,d\mathcal{H}^{N-1}=\rho_m^{-\alpha}\intav{\partial B_{r\rho_m}(X_m)}u_0\, d\mathcal{H}^{N-1}=0, \quad \forall 0<r\leq 1.
\]
Therefore, $v_0 \equiv 0$ in $B_1$ which contradicts \eqref{inteq2} properly scaled to $v_m$. 
\end{proof}

\section{Geometric estimates of the free boundary}

In this final Section we obtain further fine geometric-measure properties of the free boundary $\mathfrak{F}(u_0)$. As in Section \ref{Section Hausdorff}, here we shall work under the addition structural assumption \eqref{hip}. The first result we show concerns the local  finiteness of the $\mathcal{H}^{N-1}$-Hausdorff measure of the free boundary $\mathfrak{F}(u_0)$. 

\begin{theorem}\label{L1}
Given $\Omega'\Subset\Omega$ there exists a constant $C=C(\Omega')>0$, depending on $\Omega'$ and universal constants, such that
\[
	\Leb \left ( \mathcal{N}_{\mu}(\{u_0>0\})\cap B_\rho(X_0) \right ) \leq C\mu\rho^{N-1},
\]
wherenever, $X_0\in\Omega'\cap\partial\{u_0>0\}$, $d_0(X_0)<\frac{1}{10}\dist(\Omega',\partial\Omega)$, $\mu\ll \rho$ and $\rho$ is universally small. In particular,
$$
	\mathcal{H}^{N-1} \left ( B_\rho(X_0) \cap \mathfrak{F}(u_0) \right ) \le C \rho^{N-1}.
$$
\end{theorem}
\begin{proof}
From Theorem$\ref{haus1}$ and Theorem $\ref{uzero1}$, we have for $\varepsilon\ll 1$ 
\[
\begin{array}{c}
|\mathcal{N}_{2\mu}(\{u_\varepsilon>C_1\varepsilon^\alpha\})\cap B_{\rho}(X_0)|\leq C\mu\rho^{N-1} \\
\\
 \mbox{and} \\
\\ 
\{u_0>0\}\cap B_\rho(X_0) \subset \mathcal{N}_\mu(\{u_\varepsilon>C_1\varepsilon^\alpha\})\cap B_{\rho}(X_0). \\
\end{array}
\]
Easily we show
\[
\mathcal{N}_\mu(\{u_0>0\})\cap B_{\rho}(X_0) \subset \mathcal{N}_{2\mu}(\{u_\varepsilon>C_1\varepsilon^\alpha\})\cap B_{\rho}(X_0),
\]
which give us the estimate desired.
\end{proof}

A consequence of Theorem \ref{L1} is that the limiting region $\{u_0 > 0 \}$  has locally
finite perimeter. The key final result we will show here states that the reduced free boundary, $\redbdry \{u_0 > 0 \}$ has total measure. More importantly, we prove that around points $Z$ of the reduced free boundary, there holds
$$
	\mathcal{H}^{N-1} \left ( B_\rho(Z) \cap \mathfrak{F}(u_0) \right )  \sim \rho^{N-1}.
$$ 
In particular the free boundary has a theoretical measure outward unit vector for $\mathcal{H}^{N-1}$ almost all points in $\mathfrak{F}(u_0)$.

\begin{theorem} \label{total measure}
Given $\Omega' \Subset \Omega$, there exists a positive constant $C = C(\Omega')$, that depends only on $\Omega'$ and universal constants, such that for any ball $B_{\rho}(X_{0})$, with $\rho$ universally small, centered at a free boundary point $x_0 \in \partial\{u_0>0\}$, there holds
\[
C^{-1} \rho^{N-1} \leq \mathcal{H}^{N-1}(\partial_{red}\{u_0>0\}\cap B_\rho(X_0))\leq C\rho^{N-1}.
\]
In particular,
$$
    \mathcal{H}^{N-1} \left ( \partial\{u_0>0\} \setminus \redbdry \{u_0>0\} \right )=0.
$$
\end{theorem}
\begin{proof}
The estimate from above follows from Theorem $\ref{L1}$. It remains to verify the estimate by below.  Fixed $X_0$, let us define the normalized function $v_0 \colon B_1 \to \mathbb{R}$ by
$$
	v_0(X):=\frac{ u_0(X_0-\rho X)}{\rho^{\alpha}}.
$$
Arguing as in the proof of Theorem \eqref{hausdorff1}, for $\rho$ universally small, we conclude,
\begin{equation}\label{subharmonic red FB}
	\text{L}(v_0^{\frac{1}{\alpha}}):=\rho^{2-\alpha}\sum_{ij}f_{ij}D_{ij}(u_0^{\frac{1}{\alpha}})\geq 0 \quad \mbox{in} \quad \{v_0>0\}\cap B_1.
\end{equation}
Our next step is to furnish an appropriate special barrier. Let $\psi$ be a nonnegative smooth function in $B_1$, with $\psi \equiv 1$ in $B_{1/5}$ and $\psi \equiv 0$ outside $B_{1/4}$. Let  $\Phi$ be the solution to the following boundary value problem
$$
 \left \{
        \begin{array}{ccccc}
            \mathrm{L} \Phi & =& - \psi &\text{ in }& B_{1}\\
            \Phi &=& 0 &\text{ on }& \partial B_{1}.
        \end{array}
    \right.
$$
From classical elliptic regularity theory, $\Phi$ is smooth and, in particular, for any $0 < \alpha < 1$, 
\begin{equation} \label{TM05}
    \| \Phi  \|_{C^\alpha(B_{1/2})} \le C_1,
\end{equation}
by a universal constant $C_1>0$. Also by maximum principle $\Phi > 0$ in $B_1$ and by Hopf maximum principle,
\begin{equation} \label{TM05.5}
    f_{ij} \partial_{i}\Phi \nu_{j} \ge C_2 > 0, \quad \text{along } \partial B_1,
\end{equation}
where $\nu_{j}$ is the $j$-th coordinate of the outward normal vector to $\partial B_1(0)$. Applying generalized Gauss-Green formula, we derive
\begin{equation} \label{TM06}
    \begin{array}{lll}
        \displaystyle \int\limits_{\{v_0 > 0 \} \cap B_{1}}
        \left \{ \Phi \mathrm{L} (v_{0}^{\frac{1}{\alpha}})- v_{0}^{\frac{1}{\alpha}}\mathrm{L} \Phi \right \}
        dx &=& \displaystyle \int\limits_{\redbdry\{v_0 > 0 \}\cap B_{1}} \left \{ \Phi f_{ij}
        \partial_{i}(v_{0}^{\frac{1}{\alpha}}) - v_{0}^{\frac{1}{\alpha}}f_{ij} \partial_{i}\Phi \right \}
        \eta_{j}d\mathcal{H}^{N-1} \\
        &-& \displaystyle \int\limits_{\{v_0 > 0 \}\cap \partial
        B_{1}} v_{0}^{\frac{1}{\alpha}}f_{ij} \partial_{i}\Phi \nu_{j}d \mathcal{H}^{N-1}.
    \end{array}
\end{equation}

Since $\Phi \mathrm{L} (v_{0}^{\frac{1}{\alpha}}) \geq 0$, there holds
\begin{equation} \label{TM07}
    \int\limits_{\{v_0 > 0 \} \cap B_{1}}
    \left \{ \Phi \mathrm{L} (v_{0}^{\frac{1}{\alpha}})- v_{0}^{\frac{1}{\alpha}}\mathrm{L} \Phi \right \}
    dx
    \geq  \int\limits_{B_{1}} \psi v_{0}^{\frac{1}{\alpha}} dx \ge \int\limits_{ B_{1/5}} v_{0}^{\frac{1}{\alpha}}dx.
\end{equation}
Also from uniform gradient bounds of $v_{0}$, ellipticity and \eqref{TM05} we estimate
\begin{equation} \label{TM08}
    \left|\;\int\limits_{\redbdry\{v_0>0\} \cap
    B_{1}} \Phi f_{ij} \partial_{i}(v_{0}
^{\frac{1}{\alpha}})    \eta_{j}d\mathcal{H}^{N-1}\right| \leq C_1  \mathcal{H}^{N-1}(\redbdry\{v_0>0\}
    \cap B_{1}).
\end{equation}
In addition, clearly,
\begin{equation} \label{TM09}
    \int_{\redbdry\{v_0>0\} \cap B_{1}} v_{0}^{\frac{1}{\alpha}}
    f_{ij}\partial_{i}\Phi \eta_{j} \, d\mathcal{H}^{N-1}= 0,
\end{equation}
and by \eqref{TM05.5},
\begin{equation} \label{TM10}
    \displaystyle  \int_{\{v_0>0\}\cap \partial B_{1}} v_{0}^{\frac{1}{\alpha}}f_{ij}\partial_{i}\Phi \nu_{j}  d \mathcal{H}^{N-1}\ge 0.
\end{equation}
Combining  \eqref{TM06}, \eqref{TM07}, \eqref{TM08} and \eqref{TM09},  we deduce
\begin{equation} \label{TM11}
    \int_{B_{1/5}} v_0^{\frac{1}{\alpha}} dx \leq C_1  \mathcal{H}^{N-1}(\redbdry \{v_0>0\}
    \cap B_{1}).
\end{equation}
On the other hand, by non-degeneracy, as in proof of  Theorem \ref{int col}, there holds
\begin{equation} \label{TM12}
    \intav{B_{1/5}(0)} v_0^{\frac{1}{\alpha}} dx \ge C_3,
\end{equation}
for a positive universal constant $C_3$. Finally from \eqref{TM11} and \eqref{TM12} we conclude
$$
\mathcal{H}^{N-1}(\redbdry\{v_0>0\} \cap B_{1}) \geq c_0,
$$
for a universal constant $c_0$ and the estimate by below in proven. The total measure of the reduced free boundary follows now by classical considerations.
\end{proof}

\section*{Acknowledgments} This paper is part of the first author's PhD thesis conducted in the Department of Mathematics at Universidade Federal do Cear\'a, Brazil. Both authors would like to express their gratitude to this institution for fostering such a enjoyable and productive scientific atmosphere. The authors would like to thank the anonymous referee for a careful and thoughtful revision which greatly improved the final outcome of this work. The authors also thank Cyril Imbert for a friendly and elucidating discussion  on the heuristics of Ishii-Lions method. This work has been partially supported by CNPq-Brazil and Capes-Brazil.

\bibliographystyle{amsplain, amsalpha}

\vspace{1cm}

\noindent \textsc{Eduardo V. Teixeira} \hfill \textsc{Dami\~ao Ara\'ujo}\\
Universidade Federal do Cear\'a \hfill  Universidade Federal do Cear\'a \\
Campus of Pici - Bloco 914 \hfill Campus of Pici - Bloco 914 \\
Fortaleza - Cear\'a - Brazil  \hfill Fortaleza - Cear\'a - Brazil  \\
60.455-760 \hfill 60.455-760 \\
\texttt{teixeira@mat.ufc.br} \hfill \texttt{djunio@mat.ufc.com}

%%%%%%%%%%%%%%%%%%%%%%%%%%%%%%%%%%%%%%%%%%%%%%%%%%%%
%% END OF DOCUMENT
%%%%%%%%%%%%%%%%%%%%%%%%%%%%%%%%%%%%%%%%%%%%%%%%%

\end{document}